\documentclass{amsart}
\usepackage[utf8]{inputenc}
\usepackage[T1]{fontenc}
\usepackage[english]{babel}
\usepackage{graphicx}
\usepackage{tikz}
\usepackage{tikz-cd}
\usepackage{amssymb}
\usepackage{enumerate}
\usepackage[all]{xy}
\usepackage{comment}
\usepackage{bm}
\usepackage{verbatim}
\usepackage{amsmath,amsfonts,amssymb}
\usepackage{amsthm}
\usepackage{float}
\usepackage{hyperref}
\usepackage{xcolor}
\usepackage{marginnote}
\hypersetup{
    colorlinks=true,
    linkcolor={blue},
    citecolor={blue},
    urlcolor={blue}}
\newtheorem{te}{Theorem}[section]
\newtheorem{prop}[te]{Proposition}
\newtheorem{pb}[te]{Problem}
\newtheorem{co}[te]{Corollary}
\newtheorem{conj}[te]{Conjecture}

\newtheorem{lemma}[te]{Lemma}

\newtheorem{fact}{Fact}
\theoremstyle{definition}
\newtheorem{notation}[te]{Notation}
\newtheorem{de}[te]{Definition}
\newtheorem{ex}[te]{Example}

\theoremstyle{remark}
\newtheorem{rk}[te]{Remark}
\newlength{\plarg}
\usetikzlibrary{cd}

\renewcommand{\phi}{\varphi}

\renewcommand{\leq}{\leqslant}
\renewcommand{\geq}{\geqslant}

\newcommand{\actson}{\curvearrowright}

\newcommand {\bbZ} {{\mathbb {Z}}}


\definecolor{mygray}{gray}{0.85}
\usepackage[linecolor=black,
            backgroundcolor=mygray,
            colorinlistoftodos,
            prependcaption,
            textsize=small]{todonotes}
\usepackage{xargs}

\title{Around first-order rigidity of Coxeter groups}
\date{\today}
\thanks{Gianluca Paolini was supported by project PRIN 2022 ``Models, sets and classifications", prot. 2022TECZJA, and by INdAM Project 2024 (Consolidator grant) ``Groups, Crystals and Classifications''.}

\author{Simon Andr{\' e}}
\address{Sorbonne Universit{\' e} and Universit{\' e} Paris Cit{\' e}, CNRS, IMJ-PRG, F-75005 Paris, France.}
\email{simon.andre@imj-prg.fr}

\author{Gianluca Paolini}
\address{Department of Mathematics ``Giuseppe Peano'', University of Torino, Via Carlo Alberto 10, 10123, Italy.}
\email{gianluca.paolini@unito.it}

\begin{document}

\begin{abstract}
By the work of Sela, for any free group $F$, the Coxeter group $W_3 = \mathbb{Z}/2\mathbb{Z} \ast \mathbb{Z}/2\mathbb{Z} \ast \mathbb{Z}/2\mathbb{Z}$ is elementarily equivalent to $W_3 \ast F$, and so Coxeter groups are not closed under elementary equivalence among finitely generated groups. We study what happens after restricting to models generated by finitely many torsion elements, which we call finitely torsion-generated. We prove that if $(W,S)$ is a Coxeter system whose irreducible components are finite, affine, or non-elementary hyperbolic, and $G$ is finitely torsion-generated and elementarily equivalent to $W$, then $G$ is a Coxeter group. This combines results from \cite{MUHLHERR2022297,PS23} with the following hyperbolic result: if $W$ is a hyperbolic Coxeter group and $G$ is finitely torsion-generated and $\mathrm{AE}$-equivalent to $W$, then $G$ is isomorphic to one of finitely many Coxeter groups. We also construct two non-isomorphic hyperbolic Coxeter groups that are $\mathrm{AE}$-equivalent. We then consider first-order torsion-rigidity, meaning that $W$ is the only finitely torsion-generated model of its theory, and prove it for even hyperbolic Coxeter groups and for free products of one-ended or finite hyperbolic Coxeter groups. We conjecture that analogous phenomena hold for all Coxeter groups. Finally, we prove that elementarily equivalent even Coxeter groups are isomorphic, generalizing the analogous result for right-angled Coxeter groups from \cite{https://doi.org/10.1112/blms/bdp103}.
\end{abstract}

\maketitle

\tableofcontents

\section{Introduction}

A Coxeter group is a group $W$ which admits a presentation with generating set $S$ of the form $(ss')^{m(s, s')} = e$ for every $s,s'\in S$, where $m: S \times S \rightarrow \{1, 2, ..., \infty\}$ is a matrix such that $m(s, s') = m(s', s)$ and $m(s, s') = 1$ if and only if $s = s'$. 
The pair $(W,S)$ is called a \emph{Coxeter system}. Instead of the matrix $m$ one often uses a complete $\{1, 2, ..., \infty\}$-labelled graph $\Gamma = (S, R)$ whose edges are labelled by the numbers $m(s, s')$. Some authors omit the edge labelled $2$, and in this case we talk about the Coxeter diagram of $(W, S)$, while some authors omit the edge labelled $\infty$, in this case we talk about the Coxeter graph of $(W, S)$.
Neither being generated by involutions nor being a Coxeter group is a first-order property, so Coxeter groups are not closed under elementary equivalence. In fact, finitely generated Coxeter groups are not closed under elementary equivalence even among finitely generated groups: by the work of Sela \cite{SelX}, for every group $G$ elementarily equivalent to a (possibly cyclic) free group, the Coxeter group $W_3 = \mathbb{Z}/2\mathbb{Z} \ast \mathbb{Z}/2\mathbb{Z} \ast \mathbb{Z}/2\mathbb{Z}$ is elementarily equivalent to $W_3 \ast G$, which is not a Coxeter group. The relevant distinction for our purposes is that $W_3$ is generated by torsion elements, whereas $W_3 \ast G$ is not. The main theme of this paper is to restrict attention to finitely torsion-generated groups. Throughout the paper, all groups are assumed to be finitely generated.
Recall that a Coxeter group is said to be \emph{irreducible} if it cannot be written as a non-trivial direct product (equivalently, if its Coxeter diagram is connected). Every Coxeter group is the direct product of irreducible Coxeter groups, called its \emph{irreducible components}, that correspond to the connected components of its Coxeter diagram. A Coxeter group is said to be \emph{affine} if it is virtually abelian and infinite. Among the infinite and non-affine Coxeter groups, a class of particular interest is that of non-elementary (Gromov) hyperbolic Coxeter groups (recall that a hyperbolic group is said to be \emph{non-elementary} if it is not virtually abelian, or equivalently if it contains a free group on two generators). For instance, the triangle groups corresponding to regular tessellations of the sphere, the Euclidean plane and the hyperbolic plane are respectively finite, affine and non-elementary hyperbolic Coxeter groups.

The main result of this paper is the following theorem, where we say that a group is {\em  finitely torsion-generated} if it is \mbox{generated by finitely many torsion elements.}


        \begin{te}\label{final_main_theorem}
 Let $(W, S)$ be a Coxeter system and suppose that all the irreducible components of $(W, S)$ are either finite, affine or non-elementary hyperbolic. If $G$ is finitely torsion-generated and \mbox{elementarily equivalent to $W$, then $G$ is Coxeter.} 
	\end{te}
	
	Our Theorem~\ref{final_main_theorem} relies on three main ingredients which use different techniques:
	\begin{enumerate}[(A)]
	\item an expansion of the results from \cite{MUHLHERR2022297} on domains (cf. \cite{KVASCHUK200578});
	\item the first-order rigidity of the irreducible affine Coxeter groups proved in \cite{PS23} (which also led to a proof of their profinite rigidity);
	\item the following result on hyperbolic Coxeter groups (where we say that two groups are \emph{$\mathrm{AE}$-equivalent} if they satisfy the same first-order sentences of the form $\forall\bar{x}\exists\bar{y}\theta(\bar{x},\bar{y})$ with $\theta(\bar{x},\bar{y})$ quantifier-free).
\end{enumerate}

	\begin{te}\label{main_hyperbolic_theorem} Let $W$ be a hyperbolic Coxeter group and let $G$ be a finitely torsion-generated group. If $W$ and $G$ are $\mathrm{AE}$-equivalent, then $G$ is a hyperbolic Coxeter group. Moreover, there are only finitely many such groups $G$, up to isomorphism; denote their number by $e_W$. Furthermore, there exist hyperbolic Coxeter groups $W$ such that $e_W > 1$ (see Theorem \ref{counterexample_Cox_intro} below).
\end{te}	

	Although the proof of Theorem~\ref{final_main_theorem} uses hyperbolic geometry, the theorem also applies to non-hyperbolic Coxeter groups. Indeed, $W$ is not hyperbolic whenever $(W,S)$ has an infinite irreducible affine component other than the infinite dihedral group, or at least two distinct infinite irreducible hyperbolic components. Moreover, the methods used in \cite{PS23} to prove first-order rigidity for irreducible affine Coxeter groups are representation-theoretic. These observations motivate the following conjecture.

	\begin{conj}\label{the_main_conj} For any Coxeter group $W$, if $G$ is finitely torsion-generated and elementarily equivalent to $W$, then $G$ is Coxeter. 
\end{conj}

	As a further step toward Conjecture~\ref{the_main_conj} we also prove the following result.
	
	\begin{te}\label{reduction_to_dragon} Conjecture~\ref{the_main_conj} is true if and only if it is true for irreducible Coxeter groups.
\end{te}


	By Theorem~\ref{final_main_theorem}, Conjecture~\ref{the_main_conj} holds for irreducible spherical, affine, and hyperbolic Coxeter groups. Together, the results above provide evidence for the conjecture; the general case remains open.

\medskip

We next turn to first-order rigidity. Recall that a finitely generated group $G$ is said to be {\em first-order rigid} if $G$ is, up to isomorphism, the only finitely generated model of its first-order theory. The term first-order rigid was introduced in \cite{lub_rigidity}, where it was proved that irreducible non-uniform higher-rank characteristic-zero arithmetic lattices (e.g. $\mathrm{SL}_n(\mathbb{Z})$ for $n \geq 3$) are first-order rigid, but the notion had been considered earlier, mainly under the name of {\em quasi-axiomatizability}; see in particular \cite{doi:10.1142/S0218196703001286}. By the solution to Tarski's problem about the elementary equivalence of non-abelian free groups, non-abelian free groups are not first-order rigid. More strongly, Sela proved in \cite{SelX} that if $G$ and $H$ are non-trivial and at least one of them is not $\mathbb{Z}/2\mathbb{Z}$, and $K$ is any finitely generated group elementarily equivalent to a non-abelian free group, then $G \ast H \ast K$ is elementarily equivalent to $G \ast H$. Thus there are infinitely many isomorphism types of finitely generated groups elementarily equivalent to the universal Coxeter group $W_3 = \mathbb{Z}/2\mathbb{Z} \ast \mathbb{Z}/2\mathbb{Z} \ast \mathbb{Z}/2\mathbb{Z}$. By contrast, Theorem \ref{main_hyperbolic_theorem} gives only finitely many isomorphism types of finitely torsion-generated groups elementarily equivalent to $W_3$. Hyperbolicity alone does not imply that $W_3$ is the unique finitely torsion-generated model of its theory: Theorem \ref{main_hyperbolic_theorem} gives non-isomorphic AE-equivalent Coxeter groups (which we expect to be elementarily equivalent, assuming the conjectured $\forall\exists$-elimination of quantifiers for hyperbolic groups with torsion). This motivates the following definition, which Corollary \ref{corollary} shows is satisfied by $W_3$.


\begin{de}\label{torsion_rigidity_def}
	A finitely torsion-generated group $G$ is said to be \emph{first-order torsion-rigid} if, for every finitely torsion-generated group $G'$, the groups $G$ and $G'$ are elementarily equivalent if and only if they are isomorphic. We define in the same way the notion of \emph{$\mathrm{AE}$-torsion-rigidity} by simply considering $\mathrm{AE}$-equivalence.
\end{de}

Before moving forward, we observe that removing the assumption ``finitely'' in Definition~\ref{torsion_rigidity_def}  does not lead to interesting results. For example, let $W_2=\mathbb{Z}/2\mathbb{Z} \ast \mathbb{Z}/2\mathbb{Z}$ be the universal Coxeter group of rank $2$. Then, by Theorem \ref{virtually_cyclic} or by \cite{PS23}, we know that this group is first-order rigid. Furthermore, $W_2$ is the infinite dihedral group and so it is of the form $\mathbb{Z} \rtimes \mathbb{Z}/2\mathbb{Z}$. 
    As $W_2$ is abelian-by-finite, by \cite{Oger88} it embeds elementarily in its profinite completion $\widehat{W}_2$ which is isomorphic to $\widehat{\mathbb{Z}} \rtimes \mathbb{Z}/2\mathbb{Z}$ (cf. e.g. \cite[Proposition 2.6]{GZ11}). As $\widehat{W}_2$ is an elementary extension of $W_2$, by \cite[1.14 (3)]{MUHLHERR2022297} we have that $\widehat{W}_2$ is generated by involutions, but this group is not generated by finitely many involutions since the profinite completion of $\mathbb{Z}$ is uncountable. Similarly, although we will show below that $W_{n+2} = W_2 \ast \underbrace{\mathbb{Z}/2\mathbb{Z} \cdots \mathbb{Z}/2\mathbb{Z}}_n$ is first-order torsion-rigid, by \cite{SelX} this group is elementarily equivalent to $\widehat{W}_2 \ast \underbrace{\mathbb{Z}/2\mathbb{Z} \cdots \mathbb{Z}/2\mathbb{Z}}_n$ which is generated by involutions but is not isomorphic to $W_{n+2}$ as it is uncountable. This is why we require that $G'$ is {\em finitely} torsion-generated~in~Definition~\ref{torsion_rigidity_def}.


\medskip \noindent
First-order torsion-rigidity is a natural replacement for first-order rigidity in this setting. Indeed, by the work of Sela \cite{SelX}, the hyperbolic Coxeter group $W_3$ is elementarily equivalent to the free product $W_3\ast\mathbb{Z}$, which is not a Coxeter group. More generally, we expect $G$ and $G\ast\mathbb{Z}$ to be elementarily equivalent for any non-elementary hyperbolic group $G$ without a non-trivial normal finite subgroup; this is known for torsion-free hyperbolic groups by \cite{Sel09} and was partially generalized to groups with torsion in \cite{And19a,AF22}. This illustrates the distinction between first-order rigidity and first-order torsion-rigidity. Our result in this direction is the following (see \ref{stallings_definition} for the definition of a Stallings splitting).
	
	\begin{te}\label{main_theorem2_intro}
Let $G$ be a torsion-generated hyperbolic group, and let $\Delta$ be a reduced Stallings splitting of $G$. Suppose that the following condition holds: for every edge group $F$ of $\Delta$, the image of the natural map $N_G(F)\rightarrow \mathrm{Aut}(F)$ is equal to $\mathrm{Inn}(F)$. Then $G$ is first-order torsion-rigid (in fact, it is $\mathrm{AE}$-torsion-rigid).
\end{te}

Theorem \ref{main_theorem2_intro} does not remain true without the assumption on the edge groups of $\Delta$, even if $G$ is a hyperbolic Coxeter group. Using results from \cite{And19a}, we prove the following result (see Section \ref{example}).

	\begin{te}\label{counterexample_Cox_intro} There are finite Coxeter groups $C, A_1, B_1, A_2, B_2$, where $C$ is a special subgroup of each of the groups $A_1, B_1, A_2, B_2$, such that the Coxeter groups $W = A_1 \ast_C B_1$ and $W' = A_2 \ast_C B_2$ \mbox{are $\mathrm{AE}$-equivalent but non-isomorphic.}
\end{te}


\medskip 
The groups $W$ and $W'$ in Theorem \ref{counterexample_Cox_intro} are virtually free, hence hyperbolic. The main application of Theorem \ref{main_theorem2_intro} is the following rigidity result.

\begin{te}\label{corollary}
The following groups are $\mathrm{AE}$-torsion-rigid:
\begin{enumerate}[(1)]
    \item hyperbolic even Coxeter groups;
    \item free products of torsion-generated one-ended hyperbolic groups.
\end{enumerate}
\end{te}

In particular, hyperbolic 2-spherical Coxeter groups are $\mathrm{AE}$-torsion-rigid: such groups have Serre's property (FA), and hence are one-ended. Theorem~\ref{counterexample_Cox_intro} shows that the general situation is more varied; we therefore leave the following open problems.

\begin{pb}
Find necessary and sufficient conditions under which a hyperbolic Coxeter group is first-order torsion-rigid, or $\mathrm{AE}$-torsion-rigid.
\end{pb}

\begin{pb}
Find a geometric interpretation for the number $e_W$ from \ref{main_hyperbolic_theorem}.
\end{pb}


	We finally consider first-order torsion-rigidity for non-hyperbolic Coxeter groups. It was proved in \cite{PS23} that affine Coxeter groups are first-order rigid. On the other hand, the direct product of an infinite affine Coxeter group and a hyperbolic Coxeter group $W$ with $e_W>1$ would give a conjectural counterexample to first-order torsion-rigidity outside both the hyperbolic and affine settings. This suggests that hyperbolicity is primarily a feature of the methods rather than an intrinsic dividing line for these phenomena. In this direction, we prove the following result.
%
%

	\begin{te}\label{rigidity_Coxeter+} Let $W$  and $W'$ be even Coxeter groups. If $W$ and $W'$ are elementarily equivalent, then they are isomorphic. Furthermore, $W$ is an algebraically prime model of its first-order theory, i.e., \mbox{it embeds in every model of its theory.}
\end{te}

	Theorem~\ref{rigidity_Coxeter+} generalizes the analogous result for right-angled Coxeter groups (a subclass of even Coxeter groups) from \cite{https://doi.org/10.1112/blms/bdp103}. The ``furthermore'' clause of Theorem~\ref{rigidity_Coxeter+} cannot be strengthened: it is shown in \cite{PS23} that irreducible affine Coxeter groups do not have prime models, and some affine Coxeter groups are even, for example the triangle Coxeter group $(4,4,2)$.

 \medskip \noindent

\textbf{Notation:} throughout the paper, if $g$ denotes an element of some group $G$, we denote by $\mathrm{ad}(g)$ the inner automorphism of $G$ mapping $h\in G$ to $ghg^{-1}$.

\subsection{Acknowledgments} We are grateful to the anonymous referees for their valuable comments and suggestions.



\section{Coxeter groups}\label{prelim_coxeter}

A \emph{Coxeter group} is a group $W$ that admits a presentation of the form \[\langle s_1,\ldots,s_n \ \vert \ (s_is_j)^{m_{ij}}=1, \ \text{for all} \ i,j \rangle\] where $m_{ii}=1$ and $m_{ij}\in \mathbb{N}\cup \lbrace \infty\rbrace$ for every $1\leq i,j\leq n$ (the relation $(s_is_j)^{\infty}=1$ means that $s_is_j$ has infinite order). Note that each generator $s_i$ has order two, and that $m_{ij}=2$ if and only if $s_i$ and $s_j$ commute. The set $S=\lbrace s_1,\ldots,s_n\rbrace$ is called a \emph{Coxeter generating set}, and the pair $(W,S)$ is called a \emph{Coxeter system}. A subgroup $H\subseteq W$ is called \emph{$S$-special} if $H$ is generated by a subset of $S$, and a subgroup $H\subseteq W$ is called a \emph{reflection subgroup} if it is generated by elements from $S^W$ (i.e., if $H=\langle H\cap S^W\rangle$). We will need a result proved by Deodhar in \cite{Deodhar}, which asserts that a reflection subgroup of a Coxeter group is a Coxeter group.

A Coxeter system is even if $m_{ij}$ is \emph{even} or $\infty$ for every $1\leq i,j\leq n$, in which case it can be proved that $(W,S)$ is the only Coxeter system for $W$, and therefore it makes sense to say that $W$ is an even Coxeter group. An even Coxeter group is called a \emph{right-angled Coxeter group} if $s_is_j=s_js_i$ or $(s_is_j)^{\infty}=1$ for every $1\leq i,j\leq n$. 

Hyperbolic Coxeter groups form an important subclass of Coxeter groups. Moussong proved in his PhD thesis that a Coxeter group is hyperbolic if and only if it does not contain $\mathbb{Z}^2$, if and only if there is no pair of disjoint subsets $T_1,T_2\subseteq T$ such that $\langle T_1\rangle$ and $\langle T_2\rangle$ commute and are infinite, and there is no subset $T\subseteq S$ such that $(\langle T\rangle,T)$ is an affine Coxeter system of rank $\geq 3$ (note that affine Coxeter systems have been completely classified, and that affine Coxeter groups are virtually abelian).

A Coxeter system $(W,S)$ is said to be \emph{2-spherical} if $m_{ij}$ is finite for every $1\leq i,j\leq n$. Note that such a Coxeter group $W$ has property (FA) of Serre (and therefore $W$ is one-ended) because it is generated by a set $S$ composed of involutions such that, for every $s_1,s_2\in S$, $s_1s_2$ has finite order. 

We will need the following refinement of Stallings' decomposition theorem (we refer the reader to Subsection \ref{stallings_definition} for the definition of a Stallings splitting), which is a rephrasing of Propositions 8.8.1 and 8.8.2 in \cite{davis}. 


\begin{prop}\label{davis}
Let $(G,S)$ be a Coxeter system. If $G$ has more than one end (i.e., if $G$ is infinite and not one-ended), then there exist three proper subsets $S_0,S_1,S_2\subseteq S$ such that $S_1\cup S_2 =S$, $S_1\cap S_2=S_0$, $\langle S_0\rangle$ is finite and $G$ is the amalgamated product of $\langle S_1\rangle$ and $\langle S_2\rangle$ over $\langle S_0\rangle$ (where the injections of $\langle S_0\rangle$ into $\langle S_1\rangle$ and $\langle S_2\rangle$ are given simply by inclusion). By iterating this process, we get a reduced Stallings splitting of $G$ where each vertex group and each edge group is $S$-special (i.e., generated by some subset of $S$), and whose underlying graph is a tree. 
\end{prop}

We will also use the following lemma, which is a corollary of the previous proposition.

\begin{lemma}\label{Stallings_even}
Let $W$ be a Coxeter group, and let $\Delta$ be a reduced Stallings splitting of $W$. Then $W$ is even if and only if every vertex group of $\Delta$ is even.
\end{lemma}

\begin{proof}
If $(W,S)$ is even then by Proposition \ref{davis} every vertex group $W_v$ of $\Delta$ is an $S$-special subgroup of $W$, so $W_v$ is even. Conversely, suppose that every vertex group of $\Delta$ is even, and let $s_i,s_j\in S$. If $s_i$ and $s_j$ belong to the same vertex group $W_v$ of $\Delta$ then $m_{ij}$ is even or infinite as $W_v$ is even. If $s_i$ and $s_j$ do not belong to the same vertex group of $\Delta$ then they do not have a common fixed point in the Bass-Serre tree $T$ of $\Delta$, and therefore they generate an infinite dihedral group (otherwise the dihedral group $\langle s_i,s_j\rangle$ would be finite, hence elliptic in the tree $T$), so $m_{ij}=\infty$. 
\end{proof}

\section{Hyperbolic orbifolds}

A compact connected hyperbolic two-dimensional orbifold $\mathcal{O}$ is by definition the quotient of a convex subset $\bar{\mathcal{O}}$ of the hyperbolic plane $\mathbb{H}^2$, with geodesic boundary, by a non-elementary subgroup $G$ of $\mathrm{Isom}^{\pm}(\mathbb{H}^2)$ acting properly discontinuously (equivalently, $G$ is discrete), cocompactly and faithfully on $\bar{\mathcal{O}}$. This group $G$ is by definition the orbifold fundamental group of $\mathcal{O}$, denoted by $\pi_1^{\mathrm{orb}}(\mathcal{O})$. Equivalently, $G$ is a Fuchsian group. 

If $G$ is torsion-free, then $\mathcal{O}$ is a surface. If $G$ does not contain reflections, then the singular set of $\mathcal{O}$ is a finite collection of points ($\mathcal{O}$ has no mirrors), called the \emph{conical points} of $\mathcal{O}$, and we say that $\mathcal{O}$ is a \emph{conical orbifold}. 

Note that if we do not assume that the action $G\actson \bar{\mathcal{O}}$ is faithful, then $G$ is a finite extension of an orbifold group by the kernel of this action (which is the largest normal finite subgroup of $G$), which leads to the following definition.

\begin{de}\label{FBO}A group $G$ is called a \emph{finite-by-orbifold group} if it is an extension \[1\rightarrow F\rightarrow G \rightarrow \pi_1^{\mathrm{orb}}(\mathcal{O})\rightarrow 1\]where $\mathcal{O}$ is a conical compact connected hyperbolic two-dimensional orbifold, possibly with (geodesic) boundary, and $F$ is an arbitrary finite group called the \emph{fiber}. An \emph{extended boundary subgroup} of $G$ is the preimage in $G$ of a boundary subgroup of the orbifold fundamental group $\pi_1^{\mathrm{orb}}(\mathcal{O})$. We define in the same way the \emph{extended conical subgroups} of $G$. The orbifold $\mathcal{O}$ is called \emph{the underlying orbifold of $G$}.\end{de}

In this paper, all the orbifolds will be conical, compact, connected, hyperbolic and two-dimensional unless otherwise specified, and we will omit repeating some of these assumptions.

\begin{de}A morphism between two finite-by-orbifold groups is called a \emph{morphism of finite-by-orbifold groups} if it sends each extended boundary subgroup injectively into an extended boundary subgroup, and if it is injective on finite subgroups.
\end{de}

\begin{de}\label{pinch}Let $F\hookrightarrow G\overset{q}{\twoheadrightarrow} \pi_1(\mathcal{O})$ be a finite-by-orbifold group and let $G'$ be a group. Let $\varphi: G\rightarrow G'$ be a morphism. Let $\alpha$ be a two-sided and non-boundary-parallel simple loop on $\mathcal{O}$ representing an element of infinite order, and let $C_{\alpha}=q^{-1}(\alpha)\simeq F\rtimes\mathbb{Z}$ (well defined up to conjugacy). The curve $\alpha$ (or the subgroup $C_{\alpha}$) is said to be pinched by $\varphi$ if $\varphi(C_{\alpha})$ is finite, and the morphism $\varphi$ is said to be \emph{non-pinching} if it does not pinch any two-sided and non-boundary-parallel simple loop. Otherwise, $\varphi$ is said to be \emph{pinching}. 
\end{de}

\begin{de}\label{chi}For a conical hyperbolic orbifold $\mathcal{O}$, we denote by $\chi(\mathcal{O})$ the opposite of the orbifold Euler characteristic of $\mathcal{O}$. If $G$ is an extension of $\pi_1^{\mathrm{orb}}(\mathcal{O})$ by a finite group $F$, we define $\chi(G)=\chi(\mathcal{O})/\vert F\vert$, called the \emph{complexity} of $G$. If $\Sigma$ denotes the underlying surface of $\mathcal{O}$ and $n_1,\ldots,n_k$ denote the orders of the conical points of $\mathcal{O}$, we get \[\chi(G)=-\chi^{\mathrm{Euler}}(\Sigma)+\frac{1}{\vert F\vert}\sum_{i=1}^k \frac{1}{n_i}.\]\end{de}

The following lemma appears in \cite{And18} with the assumption that the orbifolds $\mathcal{O}$ and $\mathcal{O}'$ have nonempty boundary (see \cite[Lemma 2.36]{And18}), but here we need a version of this lemma that remains true without this restriction.

\begin{lemma}\label{complexity}
Let $\mathcal{O}$ and $\mathcal{O}'$ be conical hyperbolic orbifolds. Let $G$ and $G'$ be two finite extensions of the orbifold fundamental groups of $\mathcal{O}$ and $\mathcal{O}'$. Let $\varphi : G \rightarrow G'$ be a nonpinching morphism of finite-by-orbifold groups such that $\varphi(G)$ is not contained in an extended conical or boundary subgroup of $G'$. Then $\chi(G)\geq \chi(G')$, with equality if and only if $\varphi$ is an isomorphism.
\end{lemma}

\begin{proof}Let $F$ and $F'$ denote the fibers of $G$ and $G'$ respectively, and let $Q=G/F=\pi_1^{\mathrm{orb}}(\mathcal{O})$ and $Q'=G'/F'=\pi_1^{\mathrm{orb}}(\mathcal{O}')$. By Lemma 2.38 in \cite{And18}, $\varphi$ induces a nonpinching morphism of orbifolds $\psi: Q \rightarrow Q'$ whose image is not contained in a boundary or conical subgroup of $Q'$. By \cite[Lemma 2.34]{And18}, $\psi(Q)$ has finite index in $Q'$ (note that there is a typo in \cite[Lemma 2.34]{And18}: "infinite" should be replaced by "finite"). Let $d=[Q':\psi(Q)]$. We have $\chi(\psi(Q))=d\chi(Q')$, and we will prove that $\chi(Q)\geq \chi(\psi(Q))$ with equality if and only if $\psi$ is injective. But $\varphi(F)$ is contained in $F'$, and $\varphi$ is injective on $F$, so we get $\chi(G)\geq \chi(G')$ with equality if and only if $d=1$, $\psi$ is injective and $\varphi(F)=F'$, or equivalently if and only if $\varphi$ is an isomorphism.

\smallskip \noindent It remains to prove that $\chi(Q)\geq \chi(\psi(Q))$ with equality if and only if $\psi$ is injective. As already mentioned, in the case where $\mathcal{O}$ and $\mathcal{O}'$ have nonempty boundary, this is true by \cite[Lemma 2.36]{And18} (note that for finite extensions of conical hyperbolic orbifolds with nonempty boundary, the complexity $\chi$ defined in \cite{And18} coincides with the one given in Definition \ref{chi} above). For simplicity, let us assume that $\mathcal{O}$ and $\mathcal{O}'$ are closed, even though the argument given below would work for orbifolds with nonempty boundary as well (note that it cannot happen that $\mathcal{O}$ is closed while $\mathcal{O}'$ has nonempty boundary, or vice versa). Let $\mathcal{O}''$ denote the underlying orbifold of $\psi(Q)$, and let $\Sigma$ and $\Sigma''$ be the underlying surfaces of $\mathcal{O}$ and $\mathcal{O}''$ respectively. There is a canonical map $\pi$ from $Q$ to $\pi_1(\Sigma)$ given by forgetting the orbifold structure, i.e., by killing the conical elements. Similarly, let $\pi'': \psi(Q)\rightarrow \pi_1(\Sigma'')$ denote the canonical map. Note that the conical elements of $Q$ belong to the kernel of $\pi''\circ \psi$, so there exists a (unique) morphism $\sigma : \pi_1(\Sigma)\rightarrow \pi_1(\Sigma')$ such that $\sigma\circ \pi = \pi''\circ \psi$. Moreover, $\sigma$ is surjective. 

\smallskip \noindent First, suppose that the surfaces $\Sigma$ and $\Sigma''$ are orientable. One easily sees, for instance by considering the abelianizations of $\pi_1(\Sigma)$ and $\pi_1(\Sigma'')$, that $\Sigma$ has larger genus than $\Sigma''$. Therefore $-\chi^{\mathrm{Euler}}(\Sigma)\geq -\chi^{\mathrm{Euler}}(\Sigma'')$. Then, since $\psi$ is injective on the conical subgroups of $Q$, every conical subgroup of $\psi(Q)$ has an isomorphic preimage in $Q$. It follows from the formula given in Definition \ref{chi} that $\chi(Q)\geq \chi(\psi(Q))$. Moreover, it is clear from the formula that this is an equality if and only if $\chi^{\mathrm{Euler}}(\Sigma)=\chi^{\mathrm{Euler}}(\Sigma'')$ and every conical subgroup of $\psi(Q)$ has a unique preimage in $Q$, in which case $Q$ and $\psi(Q)$ are isomorphic, and $\psi$ is injective by the Hopf property.

\smallskip \noindent If $\Sigma$ and $\Sigma''$ are nonorientable, the same proof as in the previous paragraph works without change, using the fact that the abelianization of a nonorientable closed surface of genus $g$ is isomorphic to $\mathbb{Z}^{g-1}\times \mathbb{Z}/2\mathbb{Z}$.

\smallskip \noindent If $\Sigma$ is nonorientable of genus $g$ and $\Sigma''$ is orientable of genus $g''$ then we get $g-1\geq 2g''$, again by considering the abelianizations, and thus $-\chi^{\mathrm{Euler}}(\Sigma)=g-2\geq 2g''-2 = -\chi^{\mathrm{Euler}}(\Sigma'')$ and $\chi(Q)\geq \chi(\psi(Q))$. This is a strict inequality as otherwise $g=2g''$, which contradicts $g-1\geq 2g''$. 

\smallskip \noindent Finally, if $\Sigma$ is orientable of genus $g$ and $\Sigma''$ is nonorientable of genus $g''$ then we get $2g\geq g''$, and thus $-\chi^{\mathrm{Euler}}(\Sigma)=2g-2\geq g''-2 = -\chi^{\mathrm{Euler}}(\Sigma'')$ and $\chi(Q)\geq \chi(\psi(Q))$. Let us prove that this is a strict inequality. The surface $\Sigma''$ has an orientable cover $\Sigma''_0$ of degree two, and so $H''=\pi_1(\Sigma''_0)$ is a subgroup of $\pi_1(\Sigma'')$ of index 2. One can check that the genus $g_0''$ of $\Sigma''_0$ is equal to $g''-1$. The preimage $H$ of $H''$ in $\pi_1(\Sigma)$ by $\sigma$ is a subgroup of index 2, so $H''$ is isomorphic to the fundamental group of a closed orientable surface $\Sigma_0$ of genus $2g-1$. Therefore, if $2g-2=g''-2$ then $\sigma$ induces an isomorphism between $H$ and $H''$, and thus $\sigma$ is an isomorphism between $\pi_1(\Sigma)$ and $\pi_1(\Sigma'')$, which is a contradiction.\end{proof}




\color{black}

\section{Splittings of groups}

\subsection{Simplicial trees and splittings of groups}\label{simplicial}

We consider actions of finitely generated groups on simplicial trees. However, it is often very useful to think of trees as metric spaces rather than as combinatorial objects: we define a metric $d$ on a simplicial tree $T$ making each edge isometric to the closed interval $[0,1]$, and a simplicial action of a group $G$ on $T$ can naturally be promoted to an action by isometry on $(T,d)$. This metric is called the simplicial metric. These two points of view will be used extensively in this paper.

An element $g\in G$ or a subgroup $H\subseteq G$ is said to be \emph{elliptic} in $T$ if it fixes a point in $T$. An element $g$ that is not elliptic is \emph{hyperbolic} in $T$: there exists a unique subset $A(g)\subseteq T$ isometric to $\mathbb{R}$ that is preserved by $g$ and on which $g$ acts by translation.\color{black}

Let $G$ be a finitely generated group and let $\mathcal{A}$ be a class of groups. A \emph{splitting} or \emph{decomposition} of $G$ over $\mathcal{A}$ is an isomorphism between $G$ and the fundamental group of a graph of groups $\Delta$ with edge groups in $\mathcal{A}$. In general we omit mentioning the isomorphism and simply say that $\Delta$ is a splitting of $G$. By Bass-Serre theory \textcolor{black}{(see for instance Serre's book \cite{Serre})}, a splitting of $G$ over $\mathcal{A}$ corresponds to an action of $G$ without edge-inversions on a simplicial tree $T$, called the Bass-Serre tree of the splitting, such that edge stabilizers are in $\mathcal{A}$. 


A splitting $\Delta$ of $G$ is said to be \emph{reduced} if the following condition holds: if $e=[v,w]$ is an edge in the Bass-Serre tree of $\Delta$ such that $G_e=G_v$, then $v$ and $w$ are in the same $G$-orbit (in other words, the endpoints of $e$ must be identified in $\Delta$, \textcolor{black}{and so the image of $e$ in $\Delta$ is a loop}). 


\subsection{Stallings decompositions}\label{stallings_definition}

Here we take for $\mathcal{A}$ (defined in Subsection \ref{simplicial} above) the class of finite groups. Under the assumption that there exists a constant $K<\infty$ such that every finite subgroup of $G$ has order at most $K$, Linnell proved in \cite{Lin83} that $G$ splits as a finite graph of groups with finite edge groups and whose vertex groups do not split non-trivially over finite groups. By Stallings theorem about ends of groups, the vertex groups of this splitting are finite or one-ended. Such a splitting is called a \emph{Stallings splitting or decomposition} of $G$. It is not unique in general, but the conjugacy classes of one-ended vertex groups do not depend on a particular Stallings decomposition of $G$ (in other words, the $G$-orbits of one-ended vertex groups are the same in all Bass-Serre trees of Stallings decompositions of $G$); moreover, the conjugacy classes of finite vertex groups do not depend on a particular \emph{reduced} Stallings decomposition of $G$. A subgroup of $G$ that is conjugate to a one-ended vertex group in a Stallings decomposition of $G$ is called a \emph{one-ended factor} of $G$.

Recall that if $G$ is a hyperbolic group or a Coxeter group then $G$ has only finitely many conjugacy classes of finite subgroups, and thus $G$ has Stallings decompositions. More generally, any finitely generated group whose existential theory is contained in that of a hyperbolic or Coxeter group has Stallings decompositions (because the order of a finite subgroup is bounded in such a group).

In fact, if $G$ is a Coxeter group, Proposition \ref{davis} shows that one can read a Stallings decomposition of $G$ directly from a Coxeter presentation.

\subsection{The canonical JSJ decomposition of a hyperbolic group}\label{canonical_JSJ}

We denote by $\mathcal{Z}$ the class of groups that are virtually cyclic with infinite center. Here we take for $\mathcal{A}$ (defined in Subsection \ref{simplicial}) the class $\mathcal{Z}$. 

Any one-ended hyperbolic group has a canonical splitting over $\mathcal{Z}$, that is a splitting over $\mathcal{Z}$ that can be constructed in a natural and uniform way (see \cite{Sel97, Bow98, GL16}). In \cite{Bow98}, this decomposition is constructed from the topology of the Gromov boundary of $G$, and in \cite{GL16} it is constructed by means of the so-called tree of cylinders (see Subsection \ref{tree of cylinders} below for more details about the tree of cylinders). This decomposition is a powerful tool to study the group $G$ and its first-order theory. We mention below the features of the canonical JSJ decomposition of $G$ over $\mathcal{Z}$ that will be important for this paper, and we refer the reader to \cite{GL16} for further details.

\begin{de}\label{def_rig_vert} Let $G$ be a one-ended hyperbolic group. Let $T$ be a splitting of $G$ over $\mathcal{Z}$. A vertex $v$ of $T$ and its stabilizer $G_v$ are said to be \emph{rigid} if $G_v$ is \emph{universally elliptic}, i.e., if it is elliptic in every splitting of $G$ over $\mathcal{Z}$. Otherwise, $v$ and $G_v$ are said to be \emph{flexible}.\end{de}

The following terminology was introduced by Rips and Sela in \cite{RS97}, see also \cite[Definition 5.13]{GL16}.

\begin{de}\label{QH}Let $\Delta$ be a graph of groups and let $G$ be its fundamental group. A vertex $v$ of $\Delta$, and its stabilizer $G_v$, are said to be \emph{quadratically hanging} (denoted by QH) if $G_v$ is a finite-by-orbifold group (see Definition \ref{FBO}) with at least one extended boundary or conical subgroup, and the following condition holds: for any edge $e$ incident to $v$ in $\Delta$, the edge group $G_e$ is contained in an extended boundary subgroup of $G_v$ or is finite (in which case $G_e$ is contained in an extended conical subgroup or in the fiber of $G_v$).\end{de}


The following theorem is one of the most important results of the theory of JSJ decompositions of hyperbolic groups: it provides a description of the flexible vertices of the canonical JSJ decomposition over $\mathcal{Z}$. The result below is stated for a hyperbolic group, but it remains true in a much broader context. We refer the reader to \cite[Chapter III]{GL16} (see in particular \cite[Theorem 6.2]{GL16}).

\begin{te}\label{flexible implique rigide}Let $G$ be a one-ended hyperbolic group. If a vertex group $G_v$ of the canonical JSJ decomposition of $G$ over $\mathcal{Z}$ is flexible, then $G_v$ is quadratically hanging (see Definition \ref{QH}).\end{te}

In this paper, we refer to the canonical JSJ decomposition over $\mathcal{Z}$ constructed in \cite{GL16} by means of the tree of cylinders as \emph{the} $\mathcal{Z}$-JSJ decomposition of $G$, denoted by $\mathbf{JSJ}_G$. Proposition \ref{JSJ} below summarizes the properties of $\mathbf{JSJ}_G$ that will be useful.  

\begin{prop}\label{JSJ}Let $G$ be a one-ended hyperbolic group and let $T$ be the Bass-Serre tree of $\mathbf{JSJ}_G$.
\begin{enumerate}[(1)]
\item\textbf{Bipartition.} Every edge of $T$ joins a vertex labelled by a maximal virtually cyclic group to a vertex labelled by a group which is not virtually cyclic.
\item If $v$ is a QH vertex of $T$, then for every edge $e$ incident to $v$ in $T$, the edge group $G_e$ coincides with an extended boundary subgroup of $G_v$.
\item If $v$ is a QH vertex of $T$, then for every extended boundary subgroup $H$ of $G_v$, there exists an edge $e$ incident to $v$ in $T$ such that $G_e=H$. Moreover, the edge $e$ is unique.
\item \textbf{Acylindricity.} If an element $g\in G$ of infinite order fixes a segment of length $\geq 2$ in $T$, then this segment has length exactly 2 and its midpoint has virtually cyclic stabilizer. 
\item If $G_v$ is a flexible vertex group of $T$, then it is QH (as already mentioned in Theorem \ref{flexible implique rigide}).
\end{enumerate}

\end{prop}

\begin{rk}
Note that each virtually cyclic vertex group of $\mathbf{JSJ}_G$ is rigid (in the sense of \ref{def_rig_vert}) because it contains an edge group as a subgroup of finite index, and edge groups are universally elliptic by definition of a JSJ decomposition (see \cite{GL16}).

\smallskip \noindent Note that in the definition of a QH subgroup (see Definition \ref{QH}), every infinite incident edge group is assumed to be contained in an extended boundary subgroup of $G_v$, but is not necessarily equal to this subgroup, so the fourth point above is stronger than the definition of a QH subgroup.
\end{rk}

When the hyperbolic group $G$ is not one-ended, we consider decompositions of $G$ over the class of groups that are either finite or virtually cyclic with infinite center, denoted by $\overline{\mathcal{Z}}$. Such a decomposition of $G$ can be obtained from a reduced Stallings splitting of $G$, say $\mathbf{S}_G$, by replacing each vertex $x$ such that $G_x$ is one-ended by $\mathbf{JSJ}_{G_x}$ (the canonical JSJ decomposition of $G_x$ over $\mathcal{Z}$). The new edge groups are defined as follows: if $e=[x,y]$ is an edge of (the Bass-Serre tree of) $\mathbf{S}_G$, then $G_e$ is finite, so $G_e$ fixes a vertex $x'$ in $\mathbf{JSJ}_{G_x}$ and a vertex $y'$ in $\mathbf{JSJ}_{G_y}$, and the edge $e$ in $\mathbf{S}_G$ is simply replaced by the edge $e'=[x',y']$ in $\mathbf{JSJ}_G$. Note that the vertices $x'$ and $y'$ fixed by $G_e$ are not necessarily unique, and that the reduced Stallings splitting $\mathbf{S}_G$ is not unique in general, therefore the resulting splitting of $G$ is not unique in general, but for convenience we will still use the notation $\mathbf{JSJ}_G$ to denote one such splitting.

We now define a subgroup of $\mathrm{Aut}(G)$, called the modular group of $G$, as follows.

\begin{de}\label{modul}Let $G$ be a one-ended hyperbolic group. The \emph{modular group} $\mathrm{Mod}(G)$ of $G$ is the subgroup of $\mathrm{Aut}(G)$ consisting of all automorphisms $\sigma$ satisfying the following conditions:
\begin{enumerate}[(1)]
\item[$\bullet$]the restriction of $\sigma$ to each rigid vertex group (cf. \ref{def_rig_vert}) of $\mathbf{JSJ}_G$ coincides with the conjugacy by an element of $G$;
\item[$\bullet$]the restriction of $\sigma$ to each finite subgroup of $G$ coincides with the conjugacy by an element of $G$;
\item[$\bullet$]$\sigma$ acts trivially on the underlying graph of $\mathbf{JSJ}_G$.
\end{enumerate}
\end{de}

The following result was proved by Sela for torsion-free hyperbolic groups in \cite{Sel09} and later generalized by Reinfeldt and Weidmann to hyperbolic groups possibly with torsion (see \cite[Theorem 4.2]{RW14}). The proof relies on the so-called \emph{shortening argument} of Sela.

\begin{te}\label{SA}Let $\Gamma$ be a hyperbolic group and let $G$ be a one-ended finitely generated group. There exists a finite set $F\subseteq G\setminus\lbrace 1\rbrace$ such that, for every non-injective homomorphism $\phi : G\rightarrow\Gamma$, there exists a modular automorphism $\sigma\in\mathrm{Mod}(G)$ such that $\phi\circ\sigma$ kills an element of $F$.
\end{te}

\subsection{Centered splittings}\label{centered_construction}

\subsubsection{Definition of a centered splitting}

A centered splitting of a group $G$ is a splitting over $\overline{\mathcal{Z}}$ (the class of groups that are either finite or virtually cyclic with infinite center) that satisfies a list of nice properties inherited from the canonical JSJ decomposition of a one-ended hyperbolic group (see Subsection \ref{canonical_JSJ}), even though the group $G$ is not assumed to be one-ended and finite edge groups are allowed. The following definition is a slight variant of Definition 3.8 in \cite{And18}. 



\begin{de}\label{centered}Let $\Delta$ be a graph of groups and let $G$ be its fundamental group. Let $V$ be the set of vertices of (the underlying graph of) $\Delta$. We suppose that $\vert V\vert \geq 2$. The graph $\Delta$ is said to be \emph{centered} if the following conditions hold.
\begin{enumerate}[(1)]
\item \textbf{Strong bipartition.} The underlying graph is bipartite in a strong sense: there exists a vertex $v\in V$, called the \emph{central vertex}, such that every edge connects $v$ to a vertex in $V\setminus \lbrace v\rbrace$. Moreover, the vertex $v$ is QH. 
\item For every edge $e$, the edge group $G_e$ coincides with an extended boundary or conical subgroup of $G_v$.
\item For every extended boundary or conical subgroup $H$ of $G_v$, there exists an edge $e$ such that $G_e$ is conjugate to $H$ in $G_v$. Moreover, the edge $e$ is unique.
\item \textbf{Acylindricity.} Let $H$ be a subgroup of $G$, and suppose that $H$ is not contained in the fiber of $G_v$. If $H$ fixes a segment of length $\geq 2$ in the Bass-Serre tree of the splitting, then this segment has length exactly 2 and its endpoints are translates of $v$.
\end{enumerate}
\end{de}



\subsubsection{Construction of a centered splitting from $\mathbf{JSJ}_G$}

Let us explain how centered splittings appear naturally. The construction described below will be used several times in this paper. 

Let $G$ be a hyperbolic group. If the group $G$ is one-ended and if the canonical $\mathcal{Z}$-JSJ decomposition $\mathbf{JSJ}_G$ has at least two vertices and at least one QH vertex $v$, then the splitting $\mathbf{C}_G$ of $G$ obtained from $\mathbf{JSJ}_G$ by collapsing to a point every connected component of $\mathbf{JSJ}_G\setminus \lbrace v\rbrace$ is a centered splitting (this is an easy consequence of Proposition \ref{JSJ}).

If $G$ is not assumed to be one-ended, we describe below a similar but slightly more complicated construction. Let $\mathbf{JSJ}_G$ be a splitting of $G$ over $\overline{\mathcal{Z}}$ as defined at the end of Subsection \ref{canonical_JSJ} (this splitting is not canonical anymore in general). Suppose that $\mathbf{JSJ}_G$ has at least two vertices, and that it has at least one QH vertex $v$ whose underlying orbifold is not a closed surface (i.e., $G_v$ has at least one extended boundary or conical subgroup). Note that $G_v$ may be one-ended (which happens if and only if all the edges in $\mathbf{JSJ}_G$ incident to $v$ have finite stabilizer, or equivalently if $G_v$ appears as a vertex group in a Stallings splitting $\mathbf{S}_G$).

First, let us define a splitting $\Delta$ of $G_v$ as follows: if $G_v$ has no extended conical subgroup, $\Delta$ is simply a point; otherwise, let $C_1,\ldots, C_n$ be some representatives of the conjugacy classes of the extended conical subgroups of $G_v$, define the set of vertices of $\Delta$ as $v',v_1,\ldots,v_n$, labelled by $G_v,C_1,\ldots,C_n$ respectively, and define the edges of $\Delta$ as follows: for every vertex $v_i$, there is an edge $[v',v_i]$ labelled by $C_i$ if $i\leq n$. 

Then, let $\mathbf{JSJ}'_G$ be the splitting of $G$ obtained from $\mathbf{JSJ}_G$ by replacing the vertex $v$ by $\Delta$, with the edges incident to $\Delta$ (viewed as a subgraph of $\mathbf{JSJ}'_G$) defined as follows, where $F$ denotes the fiber of $G_v$:
\begin{enumerate}[(1)]
    \item every edge $e=[w,v]$ of $\mathbf{JSJ}_G$ such that $G_e$ is infinite is replaced by an edge $[w,v']$ in $\mathbf{JSJ}'_G$ labelled by $G_e$;
    \item for every edge $e=[w,v]$ of $\mathbf{JSJ}_G$ with $w\neq v$ such that $G_e$ is finite and such that $G_e\not\subseteq F$, $G_e$ is contained in a unique extended conical subgroup $C_i$ (or rather in a conjugate of $C_i$ in $G_v$), and we define an edge $[w,v_i]$, labelled by $G_e$;
    \item for every edge $e$ of $\mathbf{JSJ}_G$ whose endpoints are $v$ and such that $G_e$ is finite and $G_{e}\not\subseteq F$, $G_e$ is contained in a unique extended conical subgroup $C_i$ (or rather in a conjugate of $C_i$ in $G_v$), and we define an edge whose endpoints are $v_i$, labelled by $G_e$;
    \item every edge $e=[w,v]$ of $\mathbf{JSJ}_G$, with $w\neq v$, such that $G_e$ is finite and $G_{e}\subseteq F$ is replaced by an edge whose endpoints are $w$, labelled by $G_e$ (note that this operation does not disconnect the graph, since, by definition, the underlying orbifold of $G_v$ has at least one boundary or conical subgroup);
    \item every edge $e$ of $\mathbf{JSJ}_G$ whose endpoints are $v$ and such that $G_e$ is finite and $G_{e}\subseteq F$ is replaced by an edge whose endpoints are $w$, labelled by $G_e$, where $w$ is any vertex adjacent to $v$ in $\mathbf{JSJ}_G$ (note that there is at least one such vertex by assumption). 
\end{enumerate}

Finally, let $\mathbf{C}_G$ be the splitting of $G$ obtained from $\mathbf{JSJ}'_G$ by collapsing to a point every connected component of $\mathbf{JSJ}'_G\setminus \lbrace v'\rbrace$. We claim that $\mathbf{C}_G$ is a centered splitting of $G$ in the sense of Definition \ref{centered}. The fact that $v'$ is QH follows from the fact that $v$ is QH in $\mathbf{JSJ}_G$. The first condition of Definition \ref{centered} is clearly satisfied by $\mathbf{C}_G$ by construction. The second and third conditions of Definition \ref{centered} follow from the second and third items of Proposition \ref{JSJ} (for the extended boundary subgroups) and from the construction (for the extended conical subgroups). It remains to prove that the fourth condition of Definition \ref{centered} (acylindricity) is satisfied by $\mathbf{C}_G$. So let $H$ be a subgroup of $G$ that is not contained in the fiber $F$ of $G_{v'}=G_v$, suppose that $H$ fixes a segment of length $\geq 2$ in the Bass-Serre tree of $\mathbf{C}_G$, and let us prove that this segment has length exactly 2 and that its endpoints are translates of $v'$. If $H$ is infinite, the result follows immediately from the acylindricity condition in Proposition \ref{JSJ} (fourth item), so suppose that $H$ is finite. Suppose towards a contradiction that $H$ fixes a path of length two whose middle vertex is a translate of $v'$. After conjugating, one can suppose without loss of generality that the middle vertex is $v'$. Hence $H$ fixes two adjacent edges $e=[w,v']$ and $e'=[v',w']$. But by construction of the splitting $\mathbf{C}_G$, if $n$ denotes the number of conjugacy classes of extended conical subgroups of $G_{v'}$, there are exactly $n$ orbits of edges $e_1,\ldots ,e_n$ incident to $v'$ in the Bass-Serre tree of $\mathbf{C}_G$, and we have:
\begin{enumerate}[(1)]
    \item for every $1\leq i\neq j\leq n$ and every $g\in G_{v'}$, $G_{e_i}\cap gG_{e_j}g^{-1}=F$;
    \item for every $1 \leq i \leq n$ and every $g\in G_{v'}\setminus G_{e_i}$, $G_{e_i}\cap  gG_{e_i}g^{-1}=F$.
\end{enumerate}
This contradicts the fact that $H$ fixes $e$ and $e'$ with common vertex $v'$. Hence the segment fixed by $H$ is of length exactly 2 and its endpoints are translates of $v'$, which finishes the proof.

\subsection{Tree of cylinders and applications}\label{tree of cylinders}

Let $k\geq 1$ be an integer, let $G$ be a finitely generated group, and let $T$ be a splitting of $G$ over finite groups of order exactly $k$. The deformation space of $T$ is by definition the set of $G$-trees that have the same elliptic subgroups as $T$ (or equivalently, which can be obtained from $T$ by some collapse and expansion moves). In \cite{GL11}, Guirardel and Levitt construct a tree that only depends on the deformation space of $T$. This tree is called the tree of cylinders of $T$, denoted by $T_c$. We summarize below the construction of the tree of cylinders $T_c$.

First, we define an equivalence relation $\sim$ on the set of edges of $T$. We declare two edges $e$ and $e'$ to be equivalent if $G_e=G_{e'}$. Since all edge stabilizers have the same order, the union of all edges having a given stabilizer $C$ is a subtree $Y_C$, called a cylinder of $T$. In other words, $Y_C$ is $\mathrm{Fix}(C)$ (the subset of $T$ pointwise fixed by $C$). Two distinct cylinders meet in at most one point. Let us define two sets as follows:
\begin{enumerate}[(1)]
    \item[$\bullet$]$V_0$ is the set of vertices of $T$ that belong to at least two cylinders;
    \item[$\bullet$]$V_1$ is the set of cylinders of $T$.
\end{enumerate}
The \emph{tree of cylinders} $T_c$ of $T$ is the bipartite tree whose set of vertices is $V_0\sqcup V_1$, and whose edges are defined as follows: there is an edge $\varepsilon=(x, Y)$ between $x\in V_0$ and $Y\in V_1$ in $T_c$ if and only if $x\in Y$. If $Y\in V_1$ is the cylinder associated with an edge $e\in T$ (i.e., $Y=\mathrm{Fix}(G_e)$), then the vertex group $G_{Y}$ is the global stabilizer of $Y$, that is $G_Y=N_G(G_e)$.

The following lemma will play a crucial role in the proof of our main result. This lemma appears in a very similar form in \cite{And19a}.


\begin{lemma}\label{perin}Let $G$ be a finitely generated group. Let $k\geq 1$ be an integer, and let $T$ be a splitting of $G$ all of whose edge groups are of order $k$. Let $\varphi$ be an endomorphism of $G$. Suppose that $\varphi$ maps every vertex group and every edge group of $T$ isomorphically to a conjugate of itself, and that $\varphi$ is injective on the normalizer of every edge group. Then $\varphi$ is injective. Moreover, if $\varphi$ maps the normalizer of every edge group isomorphically to a conjugate of itself, then $\varphi$ is an~automorphism~of~$G$.\end{lemma}

\begin{proof}Let $T_c$ be the tree of cylinders of $T$. If $T_c$ is a point (i.e., if there is a unique cylinder in $T$), then $G$ is a vertex group of $T_c$ and $\varphi$ is an automorphism. From now on, we suppose that $T_c$ is not a point.

\smallskip \noindent First, let us define a $\varphi$-equivariant map $f:T\rightarrow T$. Let $v_1,\ldots ,v_n$ denote some representatives of the orbits of vertices of $T$. For every $1\leq i\leq n$, there exists an element $g_i\in G$ such that $\varphi(G_{v_i})=g_iG_{v_i}g_i^{-1}$. We let $f(v_i)=g_i v_i$, then we define $f$ on each vertex of $T$ by equivariance, and we define $f$ on the edges of $T$ in the obvious way (if $e=[v,w]$ is an edge of $T$, there exists a unique path $e'$ from $f(v)$ to $f(w)$ in $T$, and we let $f(e)=e'$; note that we may have $f(v)=f(w)$ \emph{a priori}, i.e., $f(e)$ may be a point).

\smallskip \noindent We claim that the map $f$ induces a $\varphi$-equivariant map $f_c : T_c \rightarrow T_c$. Indeed, for each cylinder $Y=\mathrm{Fix}(C)\subseteq T$, the image $f(Y)$ is contained in $\mathrm{Fix}(\varphi(C))$ of $T$, which is a cylinder (not a point) since $\varphi(C)$ is conjugate to $C$. If $v\in T$ belongs to two cylinders, so does $f(v)$. This allows us to define $f_c$ on the vertices of $T_c$, by sending $v\in V_0(T_c)$ to $f(v)\in V_0(T_c)$ and $Y\in V_1(T_c)$ to $f(Y)\in V_1(T_c)$. If $(v,Y)$ is an edge of $T_c$, then $f_c(v)$ and $f_c(Y)$ are adjacent in $T_c$. Hence $f_c$ maps an edge of $T_c$ to an edge of $T_c$.

\smallskip \noindent We will prove that $f_c$ does not fold any pair of edges and, therefore, that $f_c$ is injective. Assume towards a contradiction that there exist a vertex $v$ of $T_c$ and two distinct vertices $w$ and $w'$ adjacent to $v$ such that $f_c(w)=f_c(w')$. 

\smallskip \noindent First, assume that $v$ is not a cylinder. Since $T_c$ is bipartite, $w$ and $w'$ are two cylinders, associated with two edges $e$ and $e'$ of $T$. Since $f_c(w)=f_c(w')$, we have $\varphi(G_{e})=\varphi(G_{e'})$ by definition of $f_c$. But $\varphi$ is injective on $G_v$ by hypothesis, and $G_{e},G_{e'}$ are two distinct subgroups of $G_v$ (by definition of a cylinder). This is a contradiction.

\smallskip \noindent Now, assume that $v=Y_{e}$ is a cylinder. Since $f_c(w)=f_c(w')$, there exists an element $g\in G$ such that $w'=g w$. As a consequence $\varphi(g)$ belongs to $G_{f_c(w)}=\varphi(G_w)$ (here we use the assumption that $\varphi$ maps $G_w$ surjectively to a vertex group of $T$, namely a conjugate of $G_w$), so we can assume that $\varphi(g)=1$ after multiplying $g$ by an element of $G_w$. In particular, it follows that $g$ does not belong to $N_G(G_{e})$ since $\varphi$ is injective in restriction to $N_G(G_{e})$ by assumption. Then, note that $G_{e}$ is contained in $G_w$ and in $G_{w'}$ as $w,w'$ belong to the cylinder of $G_e$ in $T$, and thus $gG_{e}g^{-1}$ is contained in $gG_wg^{-1}=G_{w'}$. But $G_{e}\neq gG_{e}g^{-1}$ since $g$ does not lie in $N_G(G_{e})$, and $\varphi(G_{e})=\varphi(gG_{e}g^{-1})$ since $\varphi(g)=1$, which contradicts the injectivity of $\varphi$ on $G_{w'}$.

\smallskip \noindent Hence, $f_c$ is injective. It follows easily that $\varphi$ is injective: indeed, let $g$ be an element of $G$ such that $\varphi(g)=1$; then $f_c(g v)=f_c(v)$ for each vertex $v$ of $T_c$, so $g v=v$ for each vertex $v$ of $T_c$. But $\varphi$ is injective on vertex groups of $T_c$, so $g=1$.

\smallskip \noindent Now, suppose that $\varphi$ maps the normalizer of every edge group isomorphically to a conjugate of itself. In other words, $\varphi$ maps every vertex group of $T_c$ isomorphically to a conjugate of itself. Let us prove that $ \varphi $ is surjective. We begin by proving the surjectivity of $ f_c $. It suffices to prove the local surjectivity of $f_c$. Let $ v $ be a vertex of $ T $ and $ e $ an edge adjacent to $ v $. After translating if necessary, we can assume that $ f_c(v) = v $ and $ f_c(e) = e $. We thus have $ f_c(G_v e) = \varphi (G_v) f_c(e) = G_v f_c(e) = G_v e $. Therefore, all the translates of $ e $ by an element of $ G_v $ are in the image of $ f_c$, which proves the surjectivity of $ f_c$. It remains to prove the surjectivity of $ \varphi $. Let $ g \in G $ and let $ w $ be a vertex. There are two vertices $ v $ and $ v'$ such that $ f_c(v) = w $ and $ f_c(v') = g w $. Hence there exists $ h \in G $ such that $ v'= hv $, so $ f_c(v') = f_c(hv) = \varphi (h) w $, i.e., $ gw = \varphi (h) w $, so $ g^{- 1} \varphi (h) $ belongs to $ G_w = \varphi (G_v) $, hence $ g = \varphi (h) g'$ with $ g' \in G_v $. Therefore $ \varphi $ is surjective, so $\varphi$ is an automorphism of $G$.\end{proof}

The following lemma is a variant of Lemma \ref{perin}.

\begin{lemma}\label{perin2}Let $G,G'$ be finitely generated groups. Let $k\geq 1$ be an integer, let $T,T'$ be reduced splittings of $G,G'$ all of whose edge groups are of order $k$. Suppose that $T/G$ is a tree. Let $\varphi : G\rightarrow G'$ be a morphism. Suppose that the following assumptions hold:
\begin{enumerate}[(1)]
    \item[$\bullet$] $\varphi$ maps every vertex group (resp.\ edge group) of $T$ isomorphically to a vertex group (resp.\ edge group) of $T'$;
    \item[$\bullet$] $\varphi$ maps non-conjugate vertex groups (resp.\ edge groups) of $T$ isomorphically to non-conjugate vertex groups (resp.\ edge groups) of $T'$;
    \item[$\bullet$] $\varphi$ is injective on the normalizer of every edge group.
\end{enumerate}
Then $\varphi$ is injective.\end{lemma}

\begin{rk}
In fact, the assumption that $\varphi$ is injective on the normalizer of every edge group follows from the fact that $T/G$ is a tree and from the other assumptions satisfied by $\varphi$, but that is not the point here.
\end{rk}

\begin{proof}
We simply explain how to adapt the proof of Lemma \ref{perin}. First, we define $f: T\rightarrow T'$ in a similar way as in the proof of the previous lemma. Let us prove that $f$ induces a $\varphi$-equivariant map $f_c : T_c\rightarrow T'_c$. Let $Y=\mathrm{Fix}(G_e)\subseteq T$ be the cylinder associated with an edge group $G_e$. We claim that $\varphi(G_e)$ is an edge group of $T'$, and therefore that $f(Y)$ is contained in the cylinder of $\varphi(G_e)$ in $T'$. Write $e=[v,w]$ for some adjacent vertices $v,w$ of $T$. As $T/G$ is a reduced tree, $G_v$ and $G_w$ are non-conjugate in $G$, hence by assumption $G_v$ and $G_w$ are mapped to non-conjugate vertex groups of $T'$. It follows that the path $f(e)=[f(v),f(w)]$ is not a point, and that $\varphi(G_e)$ is contained in the stabilizer of $G'_{f(e)}$. But $G_e$ has order $k$, $\varphi$ is injective on $G_e$ and the edge groups of $T'$ have order $k$, so $\varphi(G_e)=G'_{f(e)}$. Therefore $\varphi(G_e)$ is the stabilizer of each edge in the path $f(e)$, and $f(Y)$ is contained in the cylinder of $\varphi(G_e)$ in $T'$. Then, let $v\in T$ be a vertex that belongs to two cylinders, i.e., there are two vertices $w,w'$ adjacent to $v$ such that the edges $e=[v,w]$ and $e'=[v,w']$ have distinct stabilizers in $G$. Then $f(v)$ belongs to the cylinders of $G'_{f(e)}=\varphi(G_e)$ and $G'_{f(e')}=\varphi(G_{e'})$ in $T'$, which are different as $\varphi$ is injective on $G_v$. This allows us to define $f_c$ on the vertices of $T_c$ and to prove that $f_c$ (and thus $\varphi$) are injective, as in the proof of Lemma \ref{perin}.\end{proof}

\section{Preretractions and splittings of torsion-generated groups}

\subsection{Related morphisms and preretractions}

The notions of related morphisms and preretractions were introduced by Perin in \cite{Per11} (see also the erratum \cite{Per13}). The definitions we give below are adapted to our context (we add a condition on finite subgroups, and we do not make any assumption on the images of QH vertex groups). 

\begin{de}\label{related}Let $G$ be a hyperbolic group and let $\mathbf{JSJ}_G$ be a JSJ splitting of $G$ over $\overline{\mathcal{Z}}$. Let $G'$ be a group. Two morphisms $\varphi, \psi : G\rightarrow G'$ are said to be \emph{$\mathbf{JSJ}_G$-related} if, for every finite subgroup or non-QH (with respect to $\mathbf{JSJ}_G$) vertex subgroup $H$ of $G$, there exists an element $g\in G'$ such that $\varphi_{\vert H}=\mathrm{ad}(g)\circ \psi_{\vert H}$.\end{de}

\begin{rk}\label{crucial_observation}The following observation will be crucial: if $G=\langle s_1,\ldots,s_n \ \vert \ R\rangle$ is a finite presentation, there is an existential formula $\theta(x_1,\ldots,x_n,y_1,\ldots,y_n)$ such that two morphisms $\varphi, \psi : G\rightarrow G'$ defined by $\varphi(s_i)=a_i\in G'$ and $\psi(s_i)=b_i\in G'$ for $1\leq i\leq n$ are related if and only if $G'\models \theta(a_1,\ldots,a_n,b_1,\ldots,b_n)$ (see \cite{Per11,And18} for details).\end{rk}

\begin{de}\label{pre_JSJ}Let $G$ be a hyperbolic group and let $\mathbf{JSJ}_G$ be a JSJ splitting of $G$ over $\overline{\mathcal{Z}}$. A morphism $\varphi : G\rightarrow G$ is called a \emph{$\mathbf{JSJ}_G$-preretraction} if it is $\mathbf{JSJ}_G$-related to $\mathrm{id}_G$, i.e., if it coincides with an inner automorphism on every non-QH vertex group of $\mathbf{JSJ}_G$ and on every finite subgroup of $G$. A $\mathbf{JSJ}_G$-preretraction is said to be \emph{non-degenerate} if it sends each QH vertex group isomorphically to a conjugate of itself.\end{de}

We need a similar notion for centered splittings \textcolor{black}{(see Definition \ref{centered})}.

\begin{de}\label{pre_centered}Let $G$ be a group and let $\mathbf{C}_G$ be a centered splitting. A morphism $\varphi : G\rightarrow G$ is called a \emph{$\mathbf{C}_G$-preretraction} if it coincides with an inner automorphism on every non-central vertex group of $\mathbf{C}_G$ (and thus on every finite subgroup of $G$). A $\mathbf{C}_G$-preretraction is said to be \emph{non-degenerate} if it sends the central vertex group isomorphically to a conjugate of itself.\end{de}

\begin{rk}
Here for consistency we keep the terminology "non-degenerate" used in Definition 3.6 in \cite{And18}. Note that other authors use the opposite terminology and call "degenerate" such a morphism (see \cite{GLS}).
\end{rk}

\subsection{Every non-degenerate preretraction is injective}

\begin{lemma}\label{injective0}
Let $G$ be a one-ended hyperbolic group. Then every non-degenerate $\mathbf{JSJ}_G$-preretraction is injective.
\end{lemma}

\begin{proof}
The group $G$ satisfies the assumption of \cite[Proposition 7.1]{And18}.
\end{proof}

The following lemma is a variant of Lemma \ref{injective0} where the $\mathcal{Z}$-JSJ splitting is replaced by a centered splitting. The proof of this lemma is very similar to that of \cite[Proposition 7.1]{And18} and Lemma \ref{injective0}. Note that the group $G$ below is not assumed to be one-ended. 
 
\begin{lemma}\label{injective}
Let $G$ be a finitely generated group. Let $\mathbf{C}_G$ be a centered splitting of $G$ with central vertex $v$. Then every non-degenerate $\mathbf{C}_G$-preretraction is injective.
\end{lemma}

\begin{proof}Let $T$ denote the Bass-Serre tree of the splitting $\mathbf{C}_G$ and let $\varphi$ be a non-degenerate $\mathbf{C}_G$-preretraction. First, we define a $\varphi$-equivariant map $f: T \rightarrow T$ as follows: let $v_1,\ldots ,v_n$ denote some representatives of the orbits of vertices of $T$. For every $1\leq i\leq n$, there exists an element $g_i\in G$ such that $\varphi(G_{v_i})=g_iG_{v_i}g_i^{-1}$ (by definition of a non-degenerate preretraction). We define $f(v_i)=g_i v_i$, then we define $f$ on each vertex of $T$ by equivariance, and we define $f$ on the edges of $T$ in the obvious way (if $e=[v,w]$ is an edge of $T$, there exists a unique path $e'$ from $f(v)$ to $f(w)$ in $T$, and we let $f(e)=e'$).

\smallskip \noindent In order to prove that $\varphi$ is injective, we just have to prove that $f$ is injective. Let $e=[v,w]$ be an edge of $T$. Since the centered splitting $\mathbf{C}_G$ is bipartite, we can assume without loss of generality that $v$ is a translate of the central vertex, and after translating if necessary we can assume that $f(v)=v$ and $f(w)=gw$ for some $g\in G$. By bipartition, the vertices $v$ and $w$ do not belong to the same $G$-orbit, and thus $f(v)\neq f(w)$. Hence $f(e)$ is a path of length at least one. But $G_e$ coincides with an extended boundary or conical subgroup of $G_v$, therefore it is not contained in the fiber of $G_v$ and by acylindricity of $\mathbf{C}_G$ we conclude that $f(e)$ is an edge of~$T$. 

\smallskip \noindent It remains to prove that there is no folding. If $f$ folds two edges $e=[v,w]$ and $e'=[v,w']$ then there exists an element $g\in G$ such that $w'=gw$. It follows that $f(w')=\varphi(g)f(w)=f(w)$, and thus $\varphi(g)$ belongs to $\varphi(G_w)$. Hence there is an element $h\in G_w$ such that $\varphi(g)=\varphi(h)$, so after replacing $g$ by $gh^{-1}$ one can assume that $\varphi(g)=1$. Therefore, $g$ is a hyperbolic element of translation length equal to 2. Note that $\varphi(G_e)=\varphi(G_v\cap G_w)\subseteq \varphi(G_v)\cap\varphi(G_w)=G_{f(v)}\cap G_{f(w)}=G_{f(e)}$, and thus that $\varphi(\langle G_e,G_{e'}\rangle)\subseteq G_{f(e)}$ (because $f(e)=f(e')$). Suppose towards a contradiction that $v$ is (a translate of) the central vertex of $\mathbf{C}_G$, and let $F$ denote the fiber of $G_v$. If $\vert G_e/F\vert > 2$ or $\vert G_{e'}/F\vert > 2$ then $\langle G_e,G_{e'}\rangle$ is non-elementary and thus $\varphi(\langle G_e,G_{e'}\rangle)$ is non-elementary (as $\varphi$ is injective on $G_v$), and if $\vert G_e/F\vert = 2$ and $\vert G_{e'}/F\vert = 2$ then $\langle G_e,G_{e'}\rangle$ is infinite virtually cyclic of dihedral type, hence it has finite center and thus $\varphi(\langle G_e,G_{e'}\rangle)$ is virtually cyclic with finite center. In both cases, this contradicts the fact that $\varphi(\langle G_e,G_{e'}\rangle)$ is contained in $G_{f(e)}$ with $G_{f(e)}$ in $\overline{\mathcal{Z}}$ (the class of groups that are either finite or infinite virtually cyclic with infinite center). Hence $v$ cannot be a translate of the central vertex. But $f$ folds the edges $ge=[w'=gw,gv]$ and $e'=[w',v]$ since $\varphi(g)=1$, and by bipartition of the tree $T$ the vertex $w'=gw$ is a translate of the central vertex, and we get a contradiction.\end{proof}

\subsection{Splittings of torsion-generated groups}

\begin{lemma}\label{prelim0}Let $G=A\ast_{C_A=C_B} B$ be an amalgamated product. Let $C'_B$ denote the image of $C_B$ in $B^{\mathrm{ab}}$. Let $b\in B$ be an element whose image is non-trivial in $B^{\mathrm{ab}}/C'_B$. Then $b$ cannot be a product of conjugates of elements of $A$ in $G$.\end{lemma}

\begin{proof}Let $C'_A$ denote the image of $C_A$ in $A^{\mathrm{ab}}$. Define $Q=(A^{\mathrm{ab}}/C'_A)\times (B^{\mathrm{ab}}/C'_B)$. The natural epimorphisms $\varphi_A: A\rightarrow A^{\mathrm{ab}}/C'_A$ and $\varphi_B: B\rightarrow B^{\mathrm{ab}}/C'_B$ kill $C_A$ and $C_B$ respectively, therefore we can define a (surjective) morphism $\varphi$ from $A\ast_C B$ to $Q$ by $\varphi_{\vert A}=\varphi_A$ and $\varphi_{\vert B}=\varphi_B$. By assumption the element $b\in B$ is mapped to a non-trivial element in the right direct factor, whereas any product of conjugates of elements of $A$ in $G$ is mapped to an element in the left direct factor.\end{proof} 



\begin{lemma}\label{is_a_tree}
Let $G$ be a torsion-generated group. Let $\Delta_G$ be a splitting of $G$ as a graph of groups. Then
\begin{enumerate}[(1)]
    \item the underlying graph of $\Delta_G$ is a tree, and
    \item for any QH vertex $v$ of $\Delta_G$ (see Definition \ref{QH}), the underlying orbifold of $G_v$ is orientable of genus $0$.
\end{enumerate}
\end{lemma}

\begin{proof}
If the underlying graph of $\Delta_G$ is not a tree, then $G$ splits as an HNN extension, therefore $G$ maps onto $\mathbb{Z}$. But $G$ is generated by elements of finite order, so every morphism from $G$ to a torsion-free group has trivial image. This is a contradiction.

\smallskip \noindent For the second point, notice that the fundamental groups of the torus (orientable surface of genus $1$) and of the Klein bottle (non-orientable surface of genus $2$) split as HNN extensions, respectively $\langle a,b \ \vert \ aba^{-1}=b\rangle$ and $\langle a,b \ \vert \ aba^{-1}=b^{-1}\rangle$, and thus the fundamental group of any orientable surface of genus $\geq 1$ and of any non-orientable surface of genus $\geq 2$ splits as an HNN extension. Therefore, if the underlying orbifold of $G_v$ is orientable of genus $\geq 1$ or non-orientable of genus $\geq 2$, then $G_v$ has a splitting as an HNN extension relative to the extended boundary and conical subgroups of $G_v$, and thus $G$ splits as an HNN extension and we get a contradiction as in the previous paragraph. Then, suppose that the underlying orbifold $\mathcal{O}$ of $G_v$ is non-orientable of genus $1$. Thus one can cut $\mathcal{O}$ along a simple closed curve $\gamma$ so that the connected components of $\mathcal{O}\setminus \gamma$ are respectively an orientable orbifold $\mathcal{O}'$ of genus $0$ and a Möbius band. Let $F$ denote the fiber of $G_v$. The decomposition of $\mathcal{O}$ gives rise to a splitting of $G_v/F$ as an amalgamated product $\pi_1^{\mathrm{orb}}(\mathcal{O'})\ast_{\langle \gamma^2\rangle}\langle \gamma\rangle$ relative to the boundary and conical subgroups, and thus $G$ splits as $H\ast_{F\rtimes \langle \gamma^2\rangle}(F\rtimes \langle \gamma\rangle)$ for some subgroup $H\subseteq G$. Let $\pi : F\rtimes \langle \gamma\rangle \rightarrow (F\rtimes \langle \gamma\rangle)^{\mathrm{ab}}$ be the abelianization map. Clearly, $\pi(\gamma)$ does not belong to $\pi(F\rtimes \langle \gamma^2\rangle)$. So Lemma \ref{prelim0} applies and tells us that $\gamma$ is not a product of conjugates of elements of $H$ in $G$. But by assumption $G$ is torsion-generated, so $\gamma$ is a product of elements of finite order of $G$. Moreover, every element $g\in G$ of finite order is conjugate to an element of $H$ (indeed, such an element is elliptic in the Bass-Serre tree of the splitting $H\ast_{F\rtimes \langle \gamma^2\rangle}(F\rtimes \langle \gamma\rangle)$, and every torsion element of the vertex group $F\rtimes \langle \gamma\rangle$ belongs to $F$, which is contained in the vertex group $H$). This is a contradiction. Hence the underlying orbifold of $G_v$ is necessarily orientable of genus $0$.\end{proof}


\subsection{Every preretraction of a finitely torsion-generated group is non-degenerate and injective}

\begin{lemma}\label{magical_lemma}
Let $G$ be a finitely generated group. Let $\mathbf{C}_G$ be a centered splitting of $G$ with central vertex $v$. Suppose that the underlying graph of $\mathbf{C}_G$ is a tree and that the underlying orbifold of $G_v$ is orientable of genus $0$. Then every $\mathbf{C}_G$-preretraction is non-degenerate \textcolor{black}{(see Definition \ref{pre_centered}).}
\end{lemma}

\begin{proof}Let $\lbrace w_i\rbrace_{i\in I}$ be the vertices of $\mathbf{C}_G$ distinct from the central vertex $v$. Let $T_G$ be the Bass-Serre tree of $\mathbf{C}_G$. With a slight abuse of notation we denote by $v$ a preimage of $v$ in $T_G$, and we denote by $w_i$ a preimage of $w_i$ in $T_G$ that is adjacent to $v$. 

\smallskip \noindent Let $S$ be a maximal pinched set for $\varphi$ (see Definition 7.2 in \cite{And18}), and suppose towards a contradiction that $S$ is non-empty. Let $\pi : G\rightarrow Q$ be the pinched quotient associated with $S$, and let $\psi : Q \rightarrow G$ be the unique morphism such that $\varphi = \psi \circ \pi$. Note that $Q$ has a natural decomposition as a graph of groups, denoted by $\Delta_Q$, obtained from $\mathbf{C}_G$ by replacing the central vertex $v$ by the splitting of $G_v$ dual to $S$. The vertices of $\Delta_Q$ arising from this replacement are called the new vertices, and the other vertices are still denoted by $\lbrace w_i\rbrace_{i\in I}$ with a slight abuse of notation. 

\smallskip \noindent Let $T_Q$ denote the Bass-Serre tree of $\Delta_Q$, and let $x\in T_Q$ be a new vertex that is adjacent to $w_1$ in $T_Q$. After postcomposing $\psi$ by an inner automorphism if necessary, one can assume that $\psi(Q_{w_1})=G_{w_1}$. By \cite[Lemma 7.3]{And18}, $\psi(Q_x)$ is elliptic in $T_G$. By the acylindricity condition in the definition of a centered splitting, the distance between $w_1$ and the fixed-point set of $\psi(Q_x)$ in $T_G$ is at most 1 (because $w_1$ is not a translate of $v$ and the fiber is a proper subgroup of $Q_x$). 

\smallskip \noindent \emph{First case:} if the distance equals $0$, then $\psi(Q_x)$ is contained in $\psi(Q_{w_1})=G_{w_1}$. 

\smallskip \noindent \emph{Second case:} if the distance equals $1$, then after conjugating by an element of $G_{w_1}$ if necessary, one can assume that $\psi(Q_x)\subseteq G_v$ (and still that $\psi(Q_{w_1})=G_{w_1}$). Then by Lemma \ref{complexity}, $\psi(Q_x)$ is contained in a conical or boundary subgroup of $G_v$ and thus $\psi(Q_x)$ is contained in a vertex adjacent to $v$ in $T_G$ (indeed, by definition of a centered splitting, edge groups incident to $v$ coincide with a conical subgroup or with a boundary subgroup of $G_v$). 

\smallskip \noindent Hence, in both cases, $\psi(Q_x)$ is contained in the stabilizer of a vertex adjacent to $v$ in $T_G$. After renumbering the elements of $\lbrace w_i\rbrace_{i\in I}$ if necessary, one may assume without loss of generality that $\psi(Q_x)$ is contained in $\psi(Q_{w_1})=G_{w_1}$.

\smallskip \noindent Then, recall that the underlying orbifold $\mathcal{O}$ of $G_v$ is orientable of genus $0$, and that the elements of $S$ are (by definition) not parallel to boundaries or conical points of $\mathcal{O}$. 

\smallskip \noindent It follows that the underlying orbifold of $Q_x$ has at least two boundary components or conical points, and thus there is a vertex $w_j\in\Delta_Q$ with $j\neq 1$ that is adjacent to $x$. Let $B_j$ denote the boundary or conical subgroup of $G_v$ adjacent to $w_j$. By the previous paragraph, $\varphi(B_j)$ is contained in $\varphi(G_{w_1})=G_{w_1}$. But $\varphi(B_j)\subseteq \varphi(G_{w_j})=gG_{w_j}g^{-1}$ for some $g\in G$, so $\varphi(B_j)$ fixes the path $[w_1,g\cdot w_j]$, which is not reduced to a point since $w_1$ and $w_j$ are not conjugate in $G$ (as the underlying graph of $\Delta_G$ is a tree). Moreover, the length of this path is at least 2 since $\Delta_G$ is bipartite. This is a contradiction: indeed, by definition of a centered splitting, if a subgroup of $G$ fixes a path of length 2 then it is contained in (a conjugate of) the fiber of the central vertex group $G_v$, which is not the case here since the fiber is a proper subgroup of $B_j$ and $\varphi$ is injective on $B_j$. Therefore $\varphi$ is non-pinching. 

\smallskip \noindent In fact, our argument shows that $\varphi(G_v)$ is contained in a conjugate of $G_v$. Moreover, $\varphi(G_v)$ is not contained in a boundary or conical subgroup (of this conjugate of $G_v$). Hence, by Lemma \ref{complexity}, $G_v$ is mapped isomorphically by $\varphi$ to a conjugate of $G_v$.\end{proof}


By combining Lemmas \ref{is_a_tree}, \ref{magical_lemma} and \ref{injective}, we obtain the following corollary.


\begin{co}\label{isom to a conjugate}
Let $G$ be a finitely torsion-generated group. Let $\mathbf{C}_G$ be a centered splitting of $G$ with central vertex $v$. Then every $\mathbf{C}_G$-preretraction is non-degenerate (i.e., the central vertex group $G_v$ is mapped isomorphically to a conjugate of itself), and thus injective.
\end{co}

\begin{proof}Let $\varphi$ be a $\mathbf{C}_G$-preretraction. By Lemma \ref{is_a_tree}, as $G$ is finitely torsion-generated, the underlying graph of $\mathbf{C}_G$ is a tree and the underlying orbifold of $G_v$ is orientable of genus 0. Hence $\varphi$ is non-degenerate by Lemma \ref{magical_lemma}, and $\varphi$ is injective by Lemma \ref{injective}.\end{proof}

\section{A key proposition}

\begin{prop}\label{main_theorem1}
Let $G$ be a torsion-generated  hyperbolic group, let $G'$ be a finitely torsion-generated group. Suppose that $G$ and $G'$ are AE-equivalent. Then every one-ended factor of $G$ embeds into $G'$, and every one-ended factor of $G'$ embeds into $G$. In fact, there exist two morphisms $\varphi : G\rightarrow G'$ and $\varphi' : G'\rightarrow G$ that are injective on the one-ended factors and on the finite subgroups of $G$ and $G'$ respectively, and that induce one-to-one correspondences between the conjugacy classes of one-ended factors of $G$ and $G'$, and between the conjugacy classes of \mbox{finite subgroups of $G$ and $G'$}.\end{prop}

\begin{rk}
Note that we cannot conclude directly from this proposition that $G$ and $G'$ are isomorphic, even if $G$ and $G'$ are Coxeter groups (see Section \ref{example}, where a counterexample is given).\end{rk}




\begin{proof}By \cite{And18} the group $G'$ is hyperbolic. The QH one-ended factors of $G$ and $G'$ must be treated separately from the non-QH one-ended factors. The first step of the proof consists in proving that there exist two morphisms $\varphi : G\rightarrow G'$  and $\varphi' : G'\rightarrow G$ such that the following conditions hold:
\begin{enumerate}[(1)]
\item[$\bullet$] $\varphi$ is injective on the finite subgroups of $G$ and on the non-QH one-ended factors of $G$, and $\varphi'$ is injective on the finite subgroups of $G'$ and on the non-QH one-ended factors of $G'$;
\item[$\bullet$] in fact, $\varphi$ maps every non-QH one-ended factor of $G$ isomorphically to a non-QH one-ended factor of $G'$, and $\varphi'$ maps every non-QH one-ended factor of $G'$ isomorphically to a non-QH one-ended factor of $G$;
\item[$\bullet$] moreover, $\varphi'\circ\varphi$ and $\varphi\circ\varphi'$ induce permutations of the conjugacy classes of finite subgroups and non-QH one-ended factors of $G$ and $G'$ respectively. 
\end{enumerate}

\smallskip \noindent The proof of this first step is very similar to the proof of Lemma 5.11 in \cite{And18}, so we will only explain how to adapt the proof and we refer the reader to \cite{And18} for details (note that we cannot apply \cite[Lemma 5.11]{And18} directly because it is not clear that $G$ and $G'$ are quasicores in the sense of \cite[Definition 5.11]{And18}, but the idea behind \cite[Lemma 5.11]{And18} works here without change because $G$ and $G'$ do not have non-injective preretractions (with respect to a centered splitting), by Corollary \ref{isom to a conjugate}).

\smallskip \noindent Recall that any hyperbolic group is finitely presented and has only a finite number of conjugacy classes of finite subgroups. 

\smallskip \noindent Let $G=\langle s_1,\ldots,s_n \ \vert \ \Sigma(s_1,\ldots,s_n)=1\rangle$ be a finite presentation for $G$. For any group $\Gamma$, there is a one-to-one correspondence between $\mathrm{Hom}(G,\Gamma)$ and the set \[\lbrace (\gamma_1,\ldots,\gamma_n)\in \Gamma^n, \ \Sigma(\gamma_1,\ldots,\gamma_n)=1\rbrace.\]

\smallskip \noindent Given that $G$ has only a finite number of conjugacy classes of finite subgroups, there exists a universal formula $\theta(x_1,\ldots,x_n)$ such that, for any group $\Gamma$ and for any tuple $(\gamma_1,\ldots,\gamma_n)\in \Gamma^n$, we have: $\Gamma\models \theta(\gamma_1,\ldots,\gamma_n)$ if and only if the map $ \lbrace s_1,\ldots,s_n \rbrace \rightarrow \Gamma : s_i\mapsto \gamma_i$ extends to a morphism $\varphi : G\rightarrow \Gamma$  that enjoys the following property, denoted by $(\ast)$: $\varphi$ is injective on the finite subgroups of $G$ and it maps any two non-conjugate finite subgroups of $G$ to non-conjugate subgroups of $\Gamma$. 

\smallskip \noindent Suppose towards a contradiction that every morphism $\varphi : G\rightarrow G'$ with property $(\ast)$ is non-injective on at least one of the non-QH one-ended factors of $G$. Hence, by Theorem \ref{SA}, there exists a finite set $F_i\subseteq G_i\setminus \lbrace 1\rbrace$ in each non-QH one-ended factor of $G$ such that, for every morphism $\varphi : G\rightarrow G'$ with property $(\ast)$, there exists a one-ended factor $G_i$ of $G$ and a modular automorphism $\alpha$ of $G_i$ (that is a conjugation on each finite subgroup of $G$ and on each rigid subgroup of each $G_i$), that can be naturally extended to an automorphism of $G$, still denoted by $\alpha$, such that the morphism $\psi=\varphi\circ\alpha$ kills an element in some $F_i$. Note that $\varphi$ and $\psi$ are $\mathbf{JSJ}_G$-related in the sense of Definition \ref{related}.

\smallskip \noindent The fact established in the previous paragraph is expressible via an AE sentence satisfied by $G'$ (see Remark \ref{crucial_observation}). But $G$ and $G'$ are AE-equivalent, so the following holds: for every endomorphism $\varphi$ of $G$ with property $(\ast)$, there exists an endomorphism $\psi$ of $G$ that is $\mathbf{JSJ}_G$-related to $\varphi$ and that kills one element in some $F_i$. 

\smallskip \noindent Now, take for $\varphi$ the identity of $G$: we get a $\mathbf{JSJ}_G$-preretraction $\varphi$ (see Definition \ref{pre_JSJ}) that is non-injective on $G_i$. Note that $\varphi(G_i)$ is contained in a conjugate of $G_i$, therefore one can suppose, after composing by an inner automorphism if necessary, that $\varphi_{\vert G_i}$ is a non-injective $\mathbf{JSJ}_{G_i}$-preretraction (see Definition \ref{pre_centered}) of $G_i$. By Lemma \ref{injective0}, $\varphi_{\vert G_i}$ is degenerate (recall that this means that there is a QH vertex $v$ in $\mathbf{JSJ}_{G_i}$ such that $G_v$ is not mapped isomorphically to a conjugate of itself by $\varphi_{\vert G_i}$).

\smallskip \noindent Then, let $\mathbf{C}_G$ be the centered splitting of $G$ obtained from $\mathbf{JSJ}_G$ and from the QH vertex $v$, whose construction is described in Subsection \ref{centered_construction}. Using the degenerate $\mathbf{JSJ}_{G}$-preretraction $\varphi$, one can define a $\mathbf{C}_G$-preretraction that coincides with $\varphi$ on $G_v$. This $\mathbf{C}_G$-preretraction is degenerate, which contradicts Corollary \ref{isom to a conjugate}.

\smallskip \noindent Hence, there exists a morphism $\varphi : G\rightarrow G'$ with property $(\ast)$ that is injective on the non-QH one-ended factors of $G$, and similarly there exists a morphism $\varphi' : G'\rightarrow G$ with property $(\ast)$ that is injective on the non-QH one-ended factors of $G'$. Moreover, exactly as in the proof of Lemma 5.11 in \cite{And18}, we can get two such morphisms such that $\varphi$ maps every non-QH one-ended factor of $G$ isomorphically to a non-QH one-ended factor of $G'$, and $\varphi'$ maps every non-QH one-ended factor of $G'$ isomorphically to a non-QH one-ended factor of $G$, and such that $\varphi'\circ\varphi$ and $\varphi\circ\varphi'$ induce permutations of the conjugacy classes of finite subgroups and non-QH one-ended factors of $G$ and $G'$ respectively (we refer the reader to Lemma 5.11 in \cite{And18} for details).

\smallskip \noindent It remains to deal with the QH one-ended factors. We will prove that $\varphi$ maps every QH one-ended factor of $G$ isomorphically to a QH one-ended factor of $G'$, that $\varphi'$ maps every QH one-ended factor of $G'$ isomorphically to a QH one-ended factor of $G$, and that $\varphi'\circ\varphi$ and $\varphi\circ\varphi'$ induce permutations of the conjugacy classes of QH one-ended factors of $G$ and $G'$ respectively.

\smallskip \noindent We denote by $\mathbf{S}_G$ and $\mathbf{S}_{G'}$ two Stallings decompositions of $G$ and $G'$ respectively. Let $c$ (respectively $c'$) be the smallest complexity of a QH factor of $\mathbf{S}_{G}$ (respectively $\mathbf{S}_{G'}$) \textcolor{black}{in the sense of Definition \ref{chi}}. Suppose without loss of generality that $c\leq c'$ and let $v$ be a vertex of $\mathbf{S}_{G}$ such that $G_v$ is a QH group of complexity $c$. \textcolor{black}{As $G$ has only a finite number of conjugacy classes of finite subgroups, and since $\varphi'\circ\varphi$ maps two non-conjugate finite subgroups to non-conjugate subgroups, there exists an integer $n\geq 1$ such that the endomorphism $p:=(\varphi'\circ\varphi)^n$ of $G$ coincides with an inner automorphism on each finite subgroup of $G$,}
and thus on each conical subgroup of $G_v$.

\smallskip \noindent Note that $G_v$ has at least one conical point (indeed, by Lemma \ref{is_a_tree}, the underlying orbifold of $G_v$ has genus $0$, and moreover its boundary is empty), thus the construction described in Subsection \ref{centered_construction} applies and produces a centered splitting $\mathbf{C}_G$ of $G$. We can define a $C_G$-preretraction $q$ that coincides with $p$ on $G_v$. By Corollary \ref{isom to a conjugate}, $q$ is non-degenerate, i.e., it maps $G_v$ isomorphically to a conjugate of $G_v$, therefore $p$ maps $G_v$ isomorphically to a conjugate of $G_v$. In particular $p$ is non-pinching on $G_v$, and thus $\varphi$ is non-pinching on $G_v$. It follows that $\varphi(G_v)$ is contained in a conjugate of some vertex group $G'_w$ of $\mathbf{S}_{G'}$ (by Proposition 2.31 in \cite{And18}). Clearly, this vertex group is QH, otherwise $p(G_v)$ would be contained in a non-QH vertex group of $G$. But the complexity of $G_v$ is minimal among the QH vertex groups of $G$ and $G'$, so $\chi(G'_w)\geq \chi(G_v)$, but $\chi(G_v)\geq \chi(G'_w)$ by Lemma \ref{complexity}, so $\chi(G_v)=\chi(G'_w)$ and thus $\varphi$ induces an isomorphism between $G_v$ and $G'_w$ (again by Lemma \ref{complexity}). Then, we can repeat the same process with the smallest complexity $>c$, and so on.\end{proof}

Recall that a graph of groups $\Delta$ is said to be \emph{reduced} if, for any edge of $\Delta$ with distinct endpoints, the edge group is strictly contained in the vertex groups. Equivalently, if $e=[v,w]$ is an edge in the Bass-Serre tree of $\Delta$ such that $G_e=G_v$, then $w$ is a translate of $v$ (i.e., there exists an element $g\in G$ such that $w=gv$, and so the image of $e$ in $\Delta$ is a loop).

We deduce from Proposition \ref{main_theorem1} the following corollary. 

\begin{co}\label{main_theorem1.1}
Let $G$ be a torsion-generated  hyperbolic group, let $G'$ be a finitely torsion-generated group. Let $\Delta$ and $\Delta'$ be reduced Stallings splittings of $G$ and $G'$ respectively. If $G$ and $G'$ are AE-equivalent, then there exist two morphisms $\varphi : G\rightarrow G'$ and $\varphi' : G'\rightarrow G$ that induce one-to-one correspondences between the conjugacy classes of vertex groups of $\Delta$ and $\Delta'$.\end{co}

\begin{proof}Let $\varphi$ and $\varphi'$ be the morphisms given by Proposition \ref{main_theorem1}. \textcolor{black}{Recall that each vertex group of $\Delta$ or $\Delta'$ is either one-ended or finite.} For the one-ended vertex groups, the result is a consequence of Proposition \ref{main_theorem1}. Then, let $G_v$ be a finite vertex group of $\Delta$. The group $\varphi(G_v)$ is elliptic (i.e., fixes a point) in the Bass-Serre tree $T'$ of $\Delta'$, hence there exists a vertex $w\in T'$ such that $\varphi(G_v)$ is contained in $G'_w$. We will prove that $\varphi(G_v)=G'_w$, and thus that $\varphi$ maps $G_v$ isomorphically to $G'_w$.

\smallskip \noindent First, suppose towards a contradiction that $G'_w$ is one-ended. Then, by Proposition \ref{main_theorem1}, there exists a vertex $x$ in the Bass-Serre tree $T$ of $\Delta$ whose stabilizer $G_x$ is mapped isomorphically to $G'_w$ by $\varphi$. Let $F\subseteq G_x$ be the finite subgroup of $G_x$ that is mapped isomorphically to $\varphi(G_v)$ by $\varphi$. By Proposition \ref{main_theorem1}, $F$ and $G_v$ are conjugate in $G$ as they have the same image in $G'$, therefore there exists an element $g\in G$ such that $G_v$ fixes the vertex $gx\in T$. Note that $v\neq gx$ because the stabilizer of $v$ is finite whereas the stabilizer of $x$ is infinite (as $G_x$ is mapped isomorphically by $\varphi$ to $G'_w$, which is infinite). So the path $[v,gx]$ is not reduced to a point, and it is fixed by $G_v$ since $G_v$ fixes $v$ and $gx$. Let $v_0=v,\ldots,v_n=gx$ be the consecutive vertices of this path, with $n\geq 1$. The graph of groups $\Delta$ being reduced (by assumption), $v_1$ is a translate of $v_0$ and thus $G_{v_1}=G_{v}$ since $G_v$ is finite. By iteration we get $G_{gx}=G_v$, contradicting the fact that $G_{gx}$ is infinite.

\smallskip \noindent Hence, $G'_w$ is necessarily finite. It remains to prove that $\varphi(G_v)=G'_w$. Let $F\subseteq G$ be an isomorphic preimage by $\varphi$ of $G'_w$. Let $x\in T$ be a vertex fixed by $F$ and let $y\in T'$ be a vertex fixed by $\varphi(G_x)$. Then $G'_w$ fixes the path $[w,y]$ (possibly reduced to a point). As in the previous paragraph, using the fact that $\Delta'$ is reduced, we conclude that $y$ is a translate of $w$ and that $G'_w=G'_y$. It follows that $F=G_x$. Then, by Proposition \ref{main_theorem1}, $G_v$ is conjugate in $G$ to a subgroup of $G_x$. Hence there exists an element $g\in G$ such that $G_v$ fixes the vertex $gx\in T$, and as above we conclude that $G_v=G_{gx}$. Therefore $G_v$ and $G_w$ have the same order, and thus $\varphi(G_v)=G'_w$, which proves that $G_v$ is mapped isomorphically to $G'_w$ by $\varphi$.

\smallskip \noindent Finally, the same proof applies to $\varphi'$ instead of $\varphi$.\end{proof}

We can easily deduce from Corollary \ref{main_theorem1.1} the following interesting finiteness result.

\begin{co}\label{finitely_many}
Let $G$ be a torsion-generated  hyperbolic group. Then the number of (isomorphism classes of) finitely torsion-generated groups that are AE-equivalent to $G$ is finite. 
\end{co}

\begin{rk}
If, moreover, $G$ is assumed to be a Coxeter group, we will prove in the next section that all the finitely torsion-generated groups that are AE-equivalent to $G$ are Coxeter groups as well.
\end{rk}

\begin{proof}
Let $G'$ be a finitely torsion-generated group that is AE-equivalent to $G$. Let $\Delta$ and $\Delta'$ be reduced Stallings splittings of $G$ and $G'$ respectively. By Corollary \ref{main_theorem1.1}, we know that $\Delta$ and $\Delta'$ have the same number of vertices and the same vertex groups (up to isomorphism). Moreover, $\Delta$ and $\Delta'$ are trees by Lemma \ref{is_a_tree}. The conclusion follows from the following easy observation: given a finite collection of groups $G_1,\ldots,G_n$ such that each $G_i$ has only finitely many conjugacy classes of finite subgroups, there are only finitely many graphs of groups (up to isomorphism) whose underlying graph is a tree with $n$ vertices labelled by $G_1,\ldots,G_n$ and whose edge groups are finite.\end{proof}

\section{Being a hyperbolic Coxeter group is preserved under AE-equivalence}

In this section we will prove that the property of being a hyperbolic Coxeter group is preserved under AE-equivalence among finitely torsion-generated groups.

\subsection{Injectivity results}

\begin{lemma}\label{pingpong}Let $G$ be a finitely generated group and let $T$ be a splitting of $G$. Let $G_1,\ldots,G_n$ be elliptic subgroups of $G$, and let $H$ be the subgroup of $G$ generated by $G_1,\ldots,G_n$. Suppose that, for every $i\neq j$ and $h\in H$, $\mathrm{Fix}_T(G_i)\cap \mathrm{Fix}_T(hG_jh^{-1})$ is empty. Then $H$ splits as a graph of groups whose vertex groups are conjugates of $G_1,\ldots,G_n$, and whose underlying graph is a tree.\end{lemma}

\begin{proof}Let $T_H\subseteq T$ be the minimal subtree of $H$, and let $S$ be the tree obtained from $T_H$ by collapsing $\mathrm{Fix}_{T_H}(G_i)$ to a point, for every $1\leq i\leq n$. Let $v_i$ denote the unique vertex of $S$ fixed by $G_i$. We will prove the result by induction.

\smallskip \noindent First, suppose that $n=2$. Let $p$ denote the midpoint of the path $[v_1,v_2]$, and let $S_i$ be the connected component of $S\setminus\lbrace p\rbrace$ containing $v_i$. Let $C$ denote the stabilizer of $[v_1,v_2]$. If $\vert G_1\setminus C\vert =2$ and $\vert G_2\setminus C\vert =2$ then $\langle G_1,G_2\rangle/C$ is a dihedral group, which is necessarily infinite as it acts non-elliptically on $S$, and therefore we have $\langle G_1,G_2\rangle\simeq G_1\ast_{C} G_2$. So let us suppose that $\vert G_1\setminus C\vert \geq 3$ or $\vert G_2\setminus C\vert \geq 3$. Note that $(G_1\setminus C)(S_2)\subseteq S_1$ and $(G_2\setminus C)(S_1)\subseteq S_2$, so it follows from the ping-pong lemma that $\langle G_1,G_2\rangle\simeq G_1\ast_{C} G_2$. 

\smallskip \noindent Then, suppose that $n\geq 3$, and suppose that the lemma has already been proved for $n-1$ elliptic subgroups. After renumbering the subgroups $G_1,\ldots,G_n$ if necessary, we can assume that the smallest distance in $S$ between two $H$-translates of $v_i$ and $v_j$ for $i\neq j$ is achieved for $i=1$ and $j=2$. So let $v_1$ and $hv_2$, with $h\in H$, be two vertices for which the minimum is achieved. Let $K$ be the subgroup of $H$ generated by $G_1$ and $h G_2 h^{-1}$, and let $S_K\subseteq S$ be the $K$-orbit of the path $[v_1,hv_2]$. Note that no $H$-translate of $v_i$ with $i\geq 3$ belongs to $S_K$, otherwise $v_i$ would be in the $H$-orbit of $v_1$ or $v_2$ (by minimality of the distance between $v_1$ and $hv_2$), which contradicts the assumption that, for every $i\neq j$ and $h\in H$, $\mathrm{Fix}_T(G_i)\cap \mathrm{Fix}_T(hG_jh^{-1})$ is empty.

\smallskip \noindent Let us collapse the subtree $S_K\subseteq S$ to a point $v_K$. By the induction hypothesis, $H$ splits as a graph of groups whose vertex groups are conjugates of $K,G_3,\ldots,G_n$, and whose underlying graph $\Delta$ is a tree. Furthermore, the stabilizer $K$ of $v_K$ splits as the amalgamated product of a conjugate of $G_1$ and a conjugate of $G_2$, and it is clear by construction that the stabilizer of the edge of $\Delta$ joining the vertex fixed by $K$ to the rest of the graph is contained in a conjugate of $G_1$ or $G_2$. Hence $H$ splits as a tree of groups whose vertex groups are conjugates of $G_1,\ldots,G_n$.\end{proof}

\begin{lemma}\label{normalizer_is_a_tree}Let $G$ be a finitely generated group, let $k\geq 1$ be an integer, and let $T$ be a splitting of $G$ over edge groups of order $k$. Suppose that $T/G$ is a tree. Let $G_e$ be an edge group of $T$. Then $N_G(G_e)$ splits as a tree of groups.
\end{lemma}

\begin{proof}By induction, it suffices to prove the result when $T/G$ has only one edge. Write $G=A\ast_C B$. Note that $G$ acts transitively on the edges of $T$, so $N_G(C)$ acts transitively on the edges of $\mathrm{Fix}_T(C)$. But $G$ does not act transitively on the vertices of $T$, so $N_G(C)$ does not act transitively on the vertices of $\mathrm{Fix}_T(C)$. It follows that $\mathrm{Fix}_T(C)/N_G(C)$ is simply an edge, and thus that $N_G(C)=N_A(C)\ast_C N_B(C)$. 
\end{proof}

\begin{lemma}\label{injectivity_normalizer}
Let $G$ be a finitely generated group, let $C$ be a normal subgroup of $G$. Suppose that $G$ splits as a reduced splitting $G_1\ast_C \cdots \ast_C G_n$. Let $\varphi$ be an endomorphism of $G$ such that, for every $1\leq i\leq n$, $\varphi$ maps $G_i$ injectively into a conjugate of $G_i$. Then $\varphi$ is injective.
\end{lemma}

\begin{proof}By Lemma \ref{pingpong}, $\varphi(G)$ splits as a reduced splitting $\varphi(G_1)\ast_{\varphi(C)} \cdots \ast_{\varphi(C)} \varphi(G_n)$. It is clear that an element $g\in G$ written as a normal form in $G_1\ast_C \cdots \ast_C G_n$ is mapped to an element written as a normal form in $\varphi(G_1)\ast_{\varphi(C)} \cdots \ast_{\varphi(C)} \varphi(G_n)$. Furthermore, $\varphi$ is injective on $G_1,\ldots,G_n$ by assumption. It follows that $\varphi$ is injective.
\end{proof}

\begin{lemma}\label{injectivity}
Let $G$ be a finitely generated group, let $\varphi$ be an endomorphism of $G$ and let $T$ be a reduced splitting of $G$ over finite groups. Suppose that $T/G$ is a tree. If $\varphi$ maps every vertex group $G_v$ of $T$ injectively into a conjugate of $G_v$, and every edge group $G_e$ of $T$ isomorphically to a conjugate of $G_e$, then $\varphi$ is injective.
\end{lemma}

\begin{proof}
We will prove the result by induction on the number $N$ of distinct orders of edge groups of $T$. Let us prove the induction step first. Suppose that the result is true for any splitting with at most $N\geq 1$ distinct orders of edge groups, and suppose that the number of distinct orders of edge groups of $T$ is exactly $N+1$. Let $m\geq 1$ denote the smallest order of an edge group of $T$, and collapse to a point every edge whose stabilizer has order $> m$. We get a new reduced splitting $S$ of $G$, all of whose edge groups have order $m$. Let $v$ be a vertex of $S$. Note that $G_v$ does not split non-trivially over a finite group of order $\leq m$, so $\varphi(G_v)$ fixes a vertex $w$ of $S$. Moreover, $w$ is a translate of $v$: indeed, $G_v$ contains a vertex group $G_x$ of $T$ of order $> m$, and by assumption there exists an element $g\in G$ such that $\varphi(G_x)$ fixes $gx$ in $T$, therefore $\varphi(G_x)$ fixes $gv$. Hence $\varphi(G_x)$ fixes the path $[w,gv]$, whose stabilizer has order $<m$. But $\varphi$ is injective on $G_x$, thus $w=gv$. Note that $G_v$ has a natural splitting as a tree of groups with $\leq N$ distinct orders of edge groups, so the induction hypothesis applies and shows that the endomorphism $\mathrm{ad}(g^{-1})\circ \varphi_{\vert G_v}$ of $G_v$ is injective. Finally, the induction hypothesis applies to the splitting $S$ of $G$ and shows that $\varphi$ is injective.

\smallskip \noindent It remains to prove the base case of the induction, that is the case where all the edge groups of $T$ have the same order $k$. Note that Lemma \ref{perin2} (relying on the tree of cylinders) does not apply immediately because $\varphi$ maps each vertex group injectively, but possibly not isomorphically, to a conjugate of itself. Let $G_1,\ldots,G_n$ be some representatives of the conjugacy classes of the vertex groups of $T$ such that $G_1,\ldots,G_n$ generate $G$. Define $G'_1=\varphi(G_1),\ldots,G'_n=\varphi(G_n)$.

\smallskip \noindent First, suppose that $n=2$, and define $C=G_1\cap G_2$ and $C'=G'_1\cap G'_2$. In this case, we have $\varphi(G)=\langle G'_1,G'_2\rangle\simeq G'_1\ast_{C'} G'_2$ (as in the proof of Lemma \ref{pingpong}). Note that $C$ is an edge group of $T$, so $C$ has order $k$, and $\varphi(C)$ has order $k$ as well since $\varphi(C)$ is conjugate to $C$ by assumption. Note also that $C'$ is contained in an edge group of $T$, so $\vert C'\vert\leq k$. Furthermore, $\varphi(C)=\varphi(G_1\cap G_2)\subseteq \varphi(G_1)\cap \varphi(G_2)=C'$, and therefore $\varphi(C)=C'$. Let $g\in G$ be a non-trivial element, and let $g=a_1b_1\cdots a_rb_r$ be a reduced normal form in the splitting $G=G_1\ast_C G_2$, with $r\geq 1$, $a_i\in G_1\setminus C$ (except $a_1$ that may be trivial) and $b_i\in G_2\setminus C$ (except $b_r$ that may be trivial). If $g$ belongs to $G_1$ or $G_2$ (i.e., $g=a_1$ or $g=b_1$) then $\varphi(g)\neq 1$ as by assumption $\varphi$ is injective on $G_1$ and $G_2$. If $g$ does not belong to $G_1$ or $G_2$, then $\varphi(g)=\varphi(a_1)\varphi(b_1)\cdots \varphi(a_r)\varphi(b_r)$ is a reduced normal form in $G'_1\ast_{C'} G'_2$ (because $C'=\varphi(C)$), and therefore $\varphi(g)\neq 1$. Hence $\varphi$ is injective. 

\smallskip \noindent Then, suppose that $n\geq 3$. By Lemma \ref{pingpong}, $\varphi(G)$ splits as a graph of groups $\Delta$ whose vertex groups are conjugates of $G'_1,\ldots,G'_n$, and whose underlying graph is a tree. We can assume without loss of generality that $\Delta$ is reduced, by collapsing edges if necessary. We could prove that $\varphi$ is injective by using an argument based on normal forms as we did for $n=2$, but the argument is less immediate when $n\geq 3$, so we will use Lemma \ref{perin2} instead. First, note that the edge groups of $\Delta$ have order $k$, otherwise we would have a splitting of $\varphi(G)$ over a finite group of order $<k$, contradicting the injectivity of $\varphi$ on the edge groups of $T$. Let $G'=\varphi(G)$ and let $T'$ be the Bass-Serre tree of $\Delta$. Note that $\varphi$ maps every vertex group (resp.\ edge group) of $T$ isomorphically to a vertex group (resp.\ edge group) of $T'$, and that $\varphi$ maps non-conjugate vertex groups (resp.\ edge groups) of $T$ isomorphically to non-conjugate vertex groups (resp.\ edge groups) of $T'$. It remains to prove that $\varphi$ is injective on the normalizer of every edge group: this is an immediate consequence of Lemmas \ref{normalizer_is_a_tree} and \ref{injectivity_normalizer}. Hence Lemma \ref{perin2} applies and tells us that $\varphi$ is injective.\end{proof}

\subsection{Being a hyperbolic Coxeter group is preserved under AE-equivalence}

We are ready to prove the main result of this section. Recall that if $(G,S)$ is a Coxeter system, a subgroup of $G$ is called a reflection subgroup if it is generated by elements of $S^G$. By \cite{Deodhar}, a reflection subgroup is a Coxeter group. 

\begin{te}\label{preservation}Let $G$ be a hyperbolic Coxeter group, let $G'$ be a finitely torsion-generated group. Suppose that $G$ and $G'$ are AE-equivalent. Then $G'$ is a Coxeter group. In fact, $G'$ embeds into $G$ as a reflection subgroup, and $G$ embeds into $G'$ as a reflection subgroup. Moreover, $G'$ is hyperbolic.
\end{te}

\begin{proof}Let $(G,S)$ be a Coxeter system. Let $\Delta$ and $\Delta'$ be reduced Stallings splittings of $G$ and $G'$ respectively. By Proposition \ref{davis}, we can assume that the vertex groups $G_1,\ldots,G_n$ of $\Delta$ are $S$-special. Let $G'_1,\ldots,G'_n$ be the vertex groups of $\Delta'$. By Corollary \ref{main_theorem1.1}, there exist two morphisms $\varphi : G\rightarrow G'$ and $\varphi' : G'\rightarrow G$ that induce one-to-one correspondences between the conjugacy classes of $G_1,\ldots,G_n$ and $G'_1,\ldots,G'_n$. Hence, after renumbering the $G_i$ if necessary, we can assume that $\varphi'(G'_i)=g_iG_ig_i^{-1}$ for some $g_i\in G$. By Lemma \ref{injectivity}, $\varphi\circ \varphi' : G'\rightarrow G'$ is injective, and therefore $\varphi':G'\rightarrow G$ is injective, so $G'$ is isomorphic to $\varphi'(G')$. But $\varphi'(G')$ is generated by $\varphi'(G'_1)=g_1G_1g_1^{-1},\ldots,\varphi'(G'_n)=g_nG_ng_n^{-1}$, so $\varphi'(G')$ is a reflection subgroup of $G$. Hence $\varphi'(G')$ is a Coxeter group by \cite{Deodhar}. 

\smallskip \noindent It remains to prove that $G'$ is a hyperbolic group. This is a consequence of \cite{And18}, but we can give a direct proof here. First, it is not difficult to prove that if $H=A\ast_C B$ is a hyperbolic group and $C$ is quasiconvex in $H$, then so are $A$ and $B$. Hence the one-ended factors $G_1,\ldots,G_n$ of $G$ are hyperbolic, and thus the one-ended factors $G'_1,\ldots,G'_n$ of $G'$ are hyperbolic. Then, the result follows from the combination theorem of Bestvina and Feighn \cite{BF92}.\end{proof}

\begin{rk}\label{remark_even}
Note that it follows from Lemma \ref{Stallings_even} and from the proof above that $G'$ is even if $G$ is even.
\end{rk}

\section{First-order torsion-rigidity of some hyperbolic groups}\label{FOrigidity}

Our main result in this section is the following theorem, whose proof is postponed to Subsection \ref{postponed}.

\begin{te}\label{main_theorem2}
Let $G$ be a torsion-generated hyperbolic group, and let $\Delta$ be a reduced Stallings splitting of $G$. Suppose that the following condition holds: for every edge group $F$ of $\Delta$, the image of the natural map $N_G(F)\rightarrow \mathrm{Aut}(F)$ is equal to $\mathrm{Inn}(F)$. Then $G$ is first-order torsion-rigid (in fact, AE torsion-rigid).\end{te}

Theorem \ref{main_theorem2} does not remain true without the assumption on the edge groups of $\Delta$, even if $G$ is a hyperbolic Coxeter group, as shown by the counterexample given in Section \ref{example}. 


The main application of Theorem \ref{main_theorem2} is first-order torsion-rigidity for hyperbolic even Coxeter groups and torsion-generated hyperbolic one-ended groups.

\begin{co}
The following groups are AE torsion-rigid:
\begin{enumerate}[(1)]
    \item[$\bullet$] hyperbolic even Coxeter groups;
    \item[$\bullet$] torsion-generated hyperbolic one-ended groups (and thus free products of such groups).
\end{enumerate}
\end{co}

\begin{rk}
In particular, hyperbolic 2-spherical Coxeter groups are AE torsion-rigid. Indeed, if $G$ is such a group, then it has property (FA) of Serre (and thus it is one-ended) because it is generated by a set $S$ composed of involutions such that, for every $s_1,s_2\in S$, $s_1s_2$ has finite order. 
\end{rk}

\begin{rk}
It is worth noting that Theorem \ref{main_theorem2} covers more hyperbolic Coxeter groups than Corollary \ref{corollary}. For example, the Coxeter group \[\mathrm{PGL}_2(\mathbb{Z})=\langle s_1,s_2,s_3 \ \vert \ s_1^2=s_2^2=s_3^2=(s_1s_2)^2=(s_1s_3)^3=(s_2s_3)^{\infty}=1\rangle\]is not even, nor one-ended, but Theorem \ref{main_theorem2} applies since $\mathrm{PGL}_2(\mathbb{Z})$ can be written as an amalgamated product $D_2\ast_{\mathbb{Z}/2\mathbb{Z}}D_3$ where $D_n$ denotes the dihedral group of order $2n$ and $\mathbb{Z}/2\mathbb{Z}$ has no non-trivial outer automorphism. More generally, if $A$ and $B$ are hyperbolic Coxeter groups and $C$ is a special subgroup of $A$ and $B$ with $\mathrm{Out}(C)$ trivial, then the hyperbolic Coxeter group $A\ast_C B$ is AE torsion-rigid. This is true for more general trees of hyperbolic Coxeter groups where edge groups have no non-trivial outer automorphism.\end{rk}


\begin{proof}[Proof of Corollary \ref{corollary}]
The second point is an immediate consequence of Theorem \ref{main_theorem2} since a reduced Stallings splitting of a one-ended group is simply a point, so the condition on the edge groups is empty. 

\smallskip \noindent Let us prove the first point. By \cite[Proposition 8.8.2]{davis}, an edge group $F$ in a Stallings splitting of a Coxeter group $G=\langle S\rangle$ is a special finite subgroup, which means that there exists a subset $T\subseteq S$ such that $F=\langle T\rangle$. As observed by Bahls in \cite[Proposition 5.1]{Bahls}, one can define a retraction $\rho : G\rightarrow F$ by $\rho(s)=s$ if $s\in T$ and $\rho(s)=1$ otherwise (this morphism is well-defined because every defining relation in $G$ is of the form $(ss')^m=1$ with $s,s'\in S$ and $m$ even). Now, for $g\in N_G(F)$ and for every $h\in F$, we have $ghg^{-1}=\rho(g)h\rho(g)^{-1}$, which shows that the image of the natural map $N_G(F)\rightarrow \mathrm{Aut}(F)$ is contained in $\mathrm{Inn}(F)$. Hence Theorem \ref{main_theorem2} applies.
\end{proof}

\begin{rk}Note that the existence of the retraction $\rho : G\rightarrow F$ is not true if $G$ is not assumed to be even. For instance, the finite dihedral group $G=D_3=\langle a,b \ \vert \ a^2=b^2=(ab)^3=1\rangle$ does not retract onto $\langle a \rangle$.\end{rk}

The following problem seems particularly interesting.

\begin{pb}
Find necessary and sufficient conditions (on the edge groups of a reduced Stallings splitting) under which a torsion-generated hyperbolic group (in particular a hyperbolic Coxeter group) is first-order torsion-rigid, or AE torsion-rigid.
\end{pb}

\subsection{Proof of Theorem \ref{main_theorem2}}\label{postponed}

\subsubsection{Virtually cyclic groups}

For expository purposes, we will first prove in detail (a strong version of) Theorem \ref{main_theorem2} in the particular case of virtually cyclic groups. \color{black}This result is of independent interest as we prove full first-order rigidity (not only first-order torsion-rigidity).

We will need the following well-known lemma.

\begin{lemma}
Let $G$ be an infinite virtually cyclic group. Then $G$ maps either onto $\mathbb{Z}$ or onto $D_{\infty}$ with finite kernel $C$. In the first case, $G$ is isomorphic to $C\rtimes \mathbb{Z}$ and we say that $G$ is of \emph{cyclic type}; in the latter case, $G$ is isomorphic to an amalgamated product $A\ast_C B$ with $A,B$ finite and $[A:C]=[B:C]=2$, and we say that $G$ is of \emph{dihedral type}.
\end{lemma}

\begin{proof}
Let $N$ be an infinite normal cyclic subgroup of $G$, and let $H=C_G(N)$ be its centralizer. As $Z(H)$ contains $N$, it has finite index in $H$, and so the derived subgroup $D(H)$ is finite by Schur's theorem. We get a morphism from $H$ onto the infinite finitely generated abelian group $H/D(H)$, which maps onto $\mathbb{Z}$. Let $C$ denote the finite kernel of this morphism from $H$ onto $\mathbb{Z}$, so that we have $H=\langle C,h\rangle\simeq C\rtimes \langle h\rangle$ for some $h\in H$ of infinite order. As $\mathrm{Aut}(\mathbb{Z})$ has order two, we distinguish two cases: either $H=G$ (in which case we are done) or $[G:H]=2$. In the latter case, as $C\subseteq H$ is characteristic, this subgroup is also normal in $G$. Writing $G=H\cup gH$ for some $g\in G\setminus H$, we have $g^2\in H$ and $ghg^{-1}=h^{-1}$. Hence $G/C=\langle \bar{h},\bar{g} \ \vert \ \bar{g}^2=1, \bar{g}\bar{h}\bar{g}^{-1}=\bar{h}^{-1}\rangle\simeq D_{\infty}$. Denoting by $\pi : G\rightarrow G/C$ the corresponding epimorphism and defining $A=\pi^{-1}(\bar{g})$ and $B=\pi^{-1}(\bar{g}\bar{h})$, we get a decomposition of $G$ as the amalgamated product $A\ast_C B$.\end{proof}\color{black}

Therefore, an infinite virtually cyclic group is torsion-generated if and only if it is of dihedral type.

We will also need the following easy lemma.

\color{black}
\begin{lemma}\label{non_conj_to_non_conj}
Let $G$ be a finitely presented group and let $G'$ be a group. Suppose that $G$ has only a finite number of conjugacy classes of finite subgroups. If $G$ and $G'$ are AE-equivalent, then there exists a morphism $\varphi : G \rightarrow G'$ that is injective on finite subgroups and that maps any two non-conjugate finite subgroups to non-conjugate finite subgroups.
\end{lemma}

\begin{proof}
Let $G=\langle s_1,\ldots,s_n \ \vert \ R(s_1,\ldots,s_n)=1\rangle$ be a finite presentation for $G$. Let $F_1,\ldots, F_k$ be non-conjugate finite subgroups of $G$ such that any finite subgroup of $G$ is conjugate to some $F_i$. For each $1\leq i\leq k$, let $o_i$ denote the order of $F_i$, set $F_i=\lbrace g_{i,1},\ldots,g_{i,o_i}\rbrace $ and write every $g_{i,j}$ as a word $w_{i,j}(s_1,\ldots,s_n)$. For any group $\Gamma$, as the set $\lbrace (\gamma_1,\ldots,\gamma_n)\in \Gamma^n \ \vert \ R(\gamma_1,\ldots,\gamma_n)=1\rbrace$ is in one-to-one correspondence with the set $\mathrm{Hom}(G,\Gamma)$, the following EA sentence $\theta$ is true in $\Gamma$ if and only if there exists a morphism $\varphi : G \rightarrow \Gamma$ that is injective on finite subgroups and that maps any two non-conjugate finite subgroups to non-conjugate finite subgroups ($\varphi$ is defined by mapping $s_{\ell}$ to $x_{\ell}$ for every $1\leq \ell \leq n$):\[\exists \bar{x} \ \forall g \ (R(\bar{x})=1) \bigwedge_{i=1}^k\bigwedge_{j=1}^{o_i}(w_{i,j}(\bar{x})\neq 1) \bigwedge_{i=1}^k\bigwedge_{\substack{i'=1\\ i'\neq i \\ o_i=o_{i'}}}^k\bigvee_{j=1}^{o_i}\bigwedge_{j'=1}^{o_{i'}} (gw_{i,j}(\bar{x})g^{-1}\neq w_{i',j'}(\bar{x})).\]
Clearly $G\models \theta$ since the identity of $G$ 
is injective on finite subgroups and maps any two non-conjugate finite subgroups to non-conjugate finite subgroups, so $G'\models\theta$ as well.
\end{proof}
\color{black}

\begin{te}\label{virtually_cyclic}
Let $G$ be an infinite torsion-generated virtually cyclic group. Write $G=A\ast_C B$ with $A,B$ finite and $[A:C]=[B:C]=2$, and suppose that the image of the natural map $N_G(C)\rightarrow \mathrm{Aut}(C)$ is equal to $\mathrm{Inn}(C)$. Then $G$ is first-order rigid (in fact, AE rigid).\end{te}

\begin{rk}
Note that $G$ is not only first-order torsion-rigid, it is first-order rigid. This is very specific to virtually cyclic groups and no longer holds for non-elementary hyperbolic groups. However, the strategy of the proof of Theorem \ref{virtually_cyclic} is very similar to that of Theorem \ref{main_theorem2}: \textcolor{black}{if $G'$ is a finitely generated group elementarily equivalent to $G$ (or simply AE-equivalent to $G$),} using the assumption on the edge group $C$, we prove that $G'$ surjects onto $G$, and symmetrically that $G$ surjects onto $G'$. Since virtually cyclic groups (and, more generally, hyperbolic groups) are Hopfian, we conclude that $G$ and $G'$ are isomorphic.

It is worth pointing out that virtually cyclic groups of dihedral type are not necessarily first-order rigid, as shown by Example \ref{dihedral}. Hence the condition on the edge group $C$ cannot be removed.
\end{rk}

\begin{proof}
Let $G'$ be a finitely generated group. Suppose that $G$ and $G'$ are AE-equivalent. \color{black}It is not difficult to prove that $G'$ is virtually cyclic (see for instance \cite[Proposition 2.17]{And18}). By Lemma \ref{non_conj_to_non_conj}, since $G$ and $G'$ have the same AE theory, there exists a morphism $\varphi : G\rightarrow G'$ that is injective on finite subgroups and that maps any two non-conjugate finite subgroups to non-conjugate finite subgroups, and there exists a similar morphism $\varphi' : G'\rightarrow G$.

\smallskip \noindent First, let us prove that $G'$ is of dihedral type. Suppose towards a contradiction that $G'$ is of cyclic type. Then $G'$ has a unique maximal finite subgroup $F$. Note in particular that $\varphi(A)$ and $\varphi(B)$ are contained in $F$. Note also that $\varphi'(F)$ is contained in a conjugate of $A$ or $B$. As $\varphi$ and $\varphi'$ are injective on finite subgroups, we get $\vert F\vert=\vert A\vert =\vert B\vert$. It follows that $\varphi(A)=\varphi(B)$, which contradicts the fact that $\varphi$ maps any two non-conjugate finite subgroups to non-conjugate subgroups (as clearly $A$ and $B$ are non-conjugate in $G$). Hence $G'$ is of dihedral type, so we can write $G'=A'\ast_{C'} B'$ with $A',B'$ finite and $[A':C']=[B':C']=2$.

\smallskip \noindent Note that $\varphi(A)$ is equal to a conjugate of $A'$ or $B'$, and that $\varphi(B)$ is equal to a conjugate of $B'$ or $A'$. As $\varphi$ maps $A$ and $B$ to non-conjugate subgroups (since $A$ and $B$ are non-conjugate), either $\varphi(A)$ is a conjugate of $A'$ and $\varphi(B)$ is a conjugate of $B'$ or $\varphi(A)$ is a conjugate of $B'$ and $\varphi(B)$ is a conjugate of $A'$. Hence, after renaming $A$ and $B$ to $B$ and $A$ if necessary, and after composing $\varphi'$ with an inner automorphism of $G$, one can assume that $\varphi'(A')=A$ and $\varphi'(B')=gBg^{-1}$ for some $g\in G$.\color{black}

\smallskip \noindent Let $(T,d)$ be the Bass-Serre tree of the splitting $G=A\ast_C B$ (which is simply a line), endowed with the simplicial metric $d$ \textcolor{black}{(see Subsection \ref{simplicial})}. Let $v,w$ be the unique vertices of $T$ fixed by $A$ and $B$ respectively, and let $e=[v,w]$. After composing $\varphi'$ by $\mathrm{ad}(a)$ for some element $a\in A$ if necessary, one can assume that $d(w,gw)<d(v,gw)$ (see the picture below).

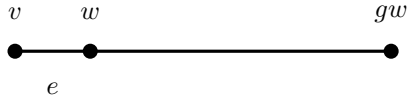
\begin{figure}[h!]
  \centering
    \begin{tikzpicture}
\draw[black,very thick] (0,0) -- (5,0);
\node[text=black] at (0,0.5) {$v$};
\node[text=black] at (0.5,-0.5) {$e$};
\node[text=black] at (1,0.5) {$w$};
\node[text=black] at (5,0.5) {$gw$};
\fill[black] (0,0) circle (0.1cm);
\fill[black] (1,0) circle (0.1cm);
\fill[black] (5,0) circle (0.1cm);
\end{tikzpicture}
\end{figure}

\smallskip \noindent Note that $\varphi'(C')=\varphi'(A'\cap B')=A\cap gBg^{-1}$, so $\varphi'(C')$ is contained in the stabilizer of the path $[v,gw]$, which contains $[v,w]$. Therefore $\varphi'(C')$ is contained in $C$. Symmetrically, $\varphi(C)$ is contained in a conjugate of $C'$. It follows that $C$ and $C'$ have the same order. Hence $\varphi'(C')=C$ and the stabilizer of the path $[v,gw]$ is $C$. 

\smallskip \noindent We will prove that $G'$ surjects onto $G$. Let us distinguish two cases.

\smallskip \noindent \emph{First case:} suppose that $g$ is hyperbolic. Then $e$ and $ge$ have the same orientation, so the edge $ge$ is contained in the path $[v,gw]$ whose stabilizer is $C$. So $g$ belongs to the normalizer $N_G(C)$.

\smallskip \noindent \emph{Second case:} suppose that $g$ is elliptic. Then $g$ fixes a point $x\in [w,gw]$ such that $d(w,x)=d(x,gw)$. If the path $[x,gw]$ is not a point, then it is fixed by $C$ (as $[x,gw]$ is contained in $[v,gw]$) and it is fixed by $gCg^{-1}$ (indeed $[w,x]$ is contained in $[v,gw]$). Hence $g$ belongs to $N_G(C)$. If $[x,gw]$ is a point then $x=gw$ and thus $gw=w$, so $g$ belongs to $B$ and thus $\varphi'$ is surjective.

\smallskip \noindent Hence, if $\varphi'$ is not surjective, then $g$ belongs to $N_G(C)$. By assumption the image of the natural map $N_G(C)\rightarrow \mathrm{Aut}(C)$ is equal to $\mathrm{Inn}(C)$, so there exists an element $c\in C$ such that $gxg^{-1}=cxc^{-1}$ for every $x\in C$. One can define a morphism $\psi' : G'\rightarrow G$ as follows: $\psi'$ coincides with $\varphi'$ on $A'$ and $\psi'$ coincides with $\mathrm{ad}(cg^{-1})\circ \varphi'$ on $B'$ \textcolor{black}{(recall that $\mathrm{ad}(cg^{-1})$ denotes the inner automorphism $x\mapsto (cg^{-1})x(cg^{-1})^{-1}$).} This morphism is well-defined because $\varphi'$ and $\mathrm{ad}(cg^{-1})\circ \varphi'$ coincide on $A'\cap B'=C'$ (since $C'$ is mapped to $C$ and $\mathrm{ad}(cg^{-1})$ coincides with the identity map on $C$). The morphism $\psi'$ is clearly surjective: indeed, $\psi'(A')=A$ and $\psi'(B')=B$. 

\smallskip \noindent We could prove that $\psi'$ is injective, and therefore an isomorphism. Instead, we will prove symmetrically that $G$ surjects onto $G'$, and conclude that $G$ and $G'$ are isomorphic by the Hopf property. We need to prove that the image of the map $N_{G'}(C')\rightarrow \mathrm{Aut}(C')$ is equal to $\mathrm{Inn}(C')$. Let $g'\in N_{G'}(C')$. Note that $\varphi'(g')$ belongs to $N_G(C)$, so there is an element $c\in C$ such that $\varphi'(g')x\varphi'(g')^{-1}=cxc^{-1}$ for every $x\in C$. There is an element $c'\in C'$ such that $c=\varphi'(c')$. Define $h=c'^{-1}g'$, and let $\theta\in \mathrm{Aut}(C')$ be such that $hyh^{-1}=\theta(y)$ for every $y\in C'$. We have $\varphi'(hyh^{-1})=\varphi'(\theta(y))=\varphi'(h)\varphi'(y)\varphi'(h)^{-1}=\varphi'(y)$ \color{black}(this last equality holds as $\varphi'(h)=\varphi'(c'^{-1}g')=c^{-1}\varphi'(g')$, and $c^{-1}\varphi'(g')$ acts trivially on $C$ by our choice of the preimage $c$ of $c'$ by $\varphi'$). \color{black}Hence $\theta$ is the identity of $C'$, which proves that $g'yg'^{-1}=c'yc'^{-1}$ for every $y\in C'$. Therefore the image of $N_{G'}(C')\rightarrow \mathrm{Aut}(C')$ is equal to $\mathrm{Inn}(C')$, and one can define a surjective morphism $\psi : G\twoheadrightarrow G'$ in the same way as we defined $\psi' : G'\twoheadrightarrow G$.

\smallskip \noindent But $G$ is Hopfian, so $\psi'\circ \psi : G \twoheadrightarrow G$ is an automorphism, and thus $\psi$ is injective. Hence $\psi$ is an isomorphism.\end{proof}

The following example shows that virtually cyclic groups of dihedral type are not necessarily first-order rigid. However, we do not know whether such an example exists among virtually cyclic Coxeter groups.

\begin{ex}\label{dihedral}Consider the following three matrices in $\mathrm{GL}_2(\bbZ/11\bbZ)$:
\[M_0=\begin{pmatrix}
    0 & 1\\
    1 & 0
  \end{pmatrix} \hspace{1cm} M_1=\begin{pmatrix}
    1 & 0\\
    0 & 2
  \end{pmatrix} \hspace{1cm} M_2=M_1^3=\begin{pmatrix}
    1 & 0\\
    0 & 8
  \end{pmatrix}\]Define $H=(\bbZ/11\bbZ)^2$ and $A,B=H\rtimes \langle M_0\rangle$, and let $\iota_i$ be the automorphism of $H$ given by the multiplication by $M_i$ for $i\in\lbrace 1,2\rbrace$. Define $G_i=\langle A,B \ \vert \ \iota_i(h)=h, \ \forall h\in H\rangle$ for $i\in\lbrace 1,2\rbrace$. Note that $H$ has index 2 in $A$ and $B$, so $H$ is normal in $G_i$ and $G_i/H$ is an infinite dihedral group (in particular, $G_i$ is virtually cyclic and hence abelian-by-finite). By \cite[Example 4.11]{GZ11}, $G_1$ and $G_2$ are not isomorphic but they have the same finite quotients, thus by \cite{Oger88} $G_1$ and $G_2$ are elementarily equivalent.
\end{ex}

\subsubsection{Proof of Theorem \ref{main_theorem2}}

We will first prove a particular case of Theorem \ref{main_theorem2}, when all the edge groups in a reduced Stallings splitting have the same order. The proof relies mainly on Lemma \ref{perin} (based on the tree of cylinders). 

\begin{te}\label{particular_case}
Let $G$ be a torsion-generated hyperbolic group. Suppose that all edge groups in a reduced Stallings splitting of $G$ have the same order, and that the following condition holds: for every edge group $F$ of a reduced Stallings splitting of $G$, the image of the natural map $N_G(F)\rightarrow \mathrm{Aut}(F)$ is equal to $\mathrm{Inn}(F)$. Then $G$ is first-order torsion-rigid (in fact, AE torsion-rigid).
\end{te}

\begin{rk}\label{rk_one_ended}In particular, this theorem proves that one-ended torsion-generated hyperbolic groups are AE torsion-rigid, since in that case the set of edges in a reduced Stallings splitting is empty.\end{rk}

\begin{proof}
Let $G'$ be a finitely torsion-generated group and suppose that $G$ and $G'$ are AE-equivalent. Let $\Delta$ and $\Delta'$ be reduced Stallings splittings of $G$ and $G'$ respectively. Let $\varphi : G\rightarrow G'$ and $\varphi' : G'\rightarrow G$ be the morphisms given by Corollary \ref{main_theorem1.1} (note that this is the only place in the proof where we use the assumption that $G$ and $G'$ are AE-equivalent). These morphisms induce one-to-one correspondences between the conjugacy classes of vertex groups of $\Delta$ and $\Delta'$. The proof of the theorem relies on the tree of cylinders defined in Subsection \ref{tree of cylinders}, more specifically on Lemma \ref{perin}. But before applying this lemma, we need to do some preparatory work.

\smallskip \noindent By assumption, all the edge groups of $\Delta$ have the same order $k$, for some integer $k\geq 1$. First, we will prove that all the edge groups of $\Delta'$ have order $k$ as well. Let us denote by $T$ and $T'$ the Bass-Serre trees of $\Delta$ and $\Delta'$ respectively. Let $e=[v,w]$ be an edge of $T'$. By Lemma \ref{is_a_tree}, the underlying graph of $\Delta'$ is a tree, moreover $\Delta'$ is reduced, therefore the vertex groups $G'_v$ and $G'_w$ are not conjugate, and thus they are mapped isomorphically by $\varphi'$ to non-conjugate vertex groups of $T$. Hence the image of the edge $[v,w]$ cannot be a point. But the edge groups of $T$ have order $k$, and $\varphi'$ is injective on finite subgroups of $G'$, therefore the edge group $G'_e$ has order $\leq k$. Then, suppose towards a contradiction that $\Delta'$ has an edge of order $<k$. Since the underlying graph of $\Delta'$ is a tree, the graph obtained by removing the interior of the edge $e$ has two connected components $X$ and $Y$, which gives rise to a splitting $G'=G'_X\ast_{G'_e} G'_Y$. \textcolor{black}{Note that for every vertex $v\in T$, the group $\varphi(G_v)$ is contained in a vertex group of $T'$ (by definition of $\varphi$), so $\varphi(G_v)$ is contained in a conjugate of $G'_X$ or $G'_Y$; moreover, since $\varphi$ is injective on the edge groups of $T$, which are of order $k$, the group $\varphi(G)$ is necessarily contained in a conjugate of $G'_X$ or $G'_Y$ (otherwise some edge group of $T$ (of order $k$ by assumption) would be mapped into a conjugate of the edge group $G'_e$, whose order is strictly less than $k$).} Without loss of generality, suppose that $\varphi(G)\subseteq G'_X$, and let $w$ be a vertex of $\Delta'$ belonging to the subtree $Y$. By assumption there is a vertex $v\in T$ such that $\varphi(G_v)$ is conjugate to $G'_w$, but no conjugate of $G'_w$ is contained in $G'_X$, which contradicts the inclusion $\varphi(G)\subseteq G'_X$. Hence the edge groups of $\Delta'$ have order exactly $k$.

\smallskip \noindent Note that $\varphi'\circ \varphi$ permutes the conjugacy classes of vertex groups of $\Delta$ and the conjugacy classes of finite subgroups of $G$, so there is an integer $m\geq 1$ such that $(\varphi'\circ \varphi)^m$ maps every vertex group of $\Delta$ and every finite subgroup of $G$ isomorphically to a conjugate of itself. Define $\psi=\varphi\circ (\varphi'\circ \varphi)^{m-1}$, so that $\varphi'\circ \psi$ maps every vertex group of $\Delta$ and every finite subgroup of $G$ isomorphically to a conjugate of itself. Similarly, one can find an integer $m'\geq 1$ such that $(\psi \circ \varphi')^{m'}$ maps every vertex group of $\Delta'$ and every finite subgroup of $G'$ isomorphically to a conjugate of itself. Define $\psi'=\psi\circ (\psi\circ \varphi')^{m'}$. From now on, let us denote by $G_1,\ldots,G_n$ and $G'_1,\ldots,G'_n$ the vertex groups of $\Delta$ and $\Delta'$ numbered in such a way that each $G_i$ is mapped isomorphically to a conjugate of $G'_i$ by $\psi$ and each $G'_i$ is mapped isomorphically to a conjugate of $G_i$ by $\psi'$.

\smallskip \noindent From now on, let us fix an edge $e$ of $T$. The edge group $G_e$ is mapped isomorphically to an edge group $G'_{e'}$ of $T'$, and thus the normalizer $N=N_G(G_e)$ is mapped (not necessarily isomorphically) to $N'=N_{G'}(G'_{e'})$ by $\psi$ and $N'$ is mapped (not necessarily isomorphically) to a conjugate of $N$ by $\psi'$ (indeed, $G_e$ is finite and $\psi'\circ \psi(G_e)$ is conjugate to $G_e$). We can assume without loss of generality that $\psi'(N')\subseteq N$ (after replacing $\psi'$ with $\mathrm{ad}(g)\circ \psi'$ for some $g\in G$ if necessary). We will modify $\psi$ and $\psi'$ so that $N$ is mapped isomorphically to $N'$ and vice versa, in order to be in a position to apply Lemma \ref{perin}.

\smallskip \noindent Recall that $N$ is the stabilizer of the cylinder $Y=\mathrm{Fix}(G_e)\subseteq T$ of $G_e$, and $N'$ is the stabilizer of the cylinder $Y'=\mathrm{Fix}(G'_{e'})\subseteq T'$. Hence the actions $N\actson Y$ and $N'\actson Y'$ give rise to splittings of $N$ and $N'$ as graphs of groups. 

\smallskip \noindent Let $v,w\in Y$ be two vertices. Suppose that their stabilizers $N_v=G_v\cap N$ and $N_w=G_w\cap N$ are non-conjugate in $N$, and let us prove that $\psi(N_v)$ and $\psi(N_w)$ are non-conjugate in $N'$. There exist two vertices $x,y$ of $Y'$ such that $\psi(G_v)=G'_x$ and $\psi(G_w)=G'_y$. We have $\psi(N_v)=\psi(G_v\cap N)\subseteq \psi(G_v)\cap \psi(N)\subseteq G'_x\cap N'=N'_x$. Let us prove that the inclusion $\psi(N_v)\subseteq N'_x$ is in fact an equality. Let $g'\in N'_x=G'_x\cap N'$, and let $g\in G_v$ be such that $\psi(g)=g'$. As $v$ belongs to the cylinder $Y$, the edge group $G_e$ is contained in $G_v$, and $gG_eg^{-1}$ is also contained in $G_v$. If $gG_eg^{-1}\neq G_e$ then, by injectivity of $\psi$ on $G_v$, we get $g'G'_{e'}g'^{-1}\neq G'_{e'}$, which contradicts the fact that $g'$ belongs to $N'$. Hence $g$ belongs to $G_v\cap N$, which proves that $\psi(N_v)=N'_x$. Similarly, $\psi(N_w)=N'_y$. Now, suppose towards a contradiction that $N'_x$ and $N'_y$ are conjugate in $N'$, and let $n'\in N'$ be such that $n'N'_xn'^{-1}=N'_y$. If $N'_x$ has order $k$ then $N'_x=N'_y=G'_{e'}$ and so $N_v=N_w=G_e$, which contradicts the fact that $N_v$ and $N_w$ are non-conjugate. Hence $\vert N'_x\vert = \vert N'_y\vert >k$, so $x$ is the only point of $T'$ fixed by $N'_x$ and $y$ is the only point of $T'$ fixed by $N'_y$, and it follows that $n'x=y$, and thus $n'G'_xn'^{-1}=G'_y$. So there is an element $g\in G$ such that $gG_vg^{-1}=G_w$. Applying $\psi$ to this equality, we obtain $\psi(g)G'_x\psi(g)^{-1}=G'_y$, therefore $n'^{-1}\psi(g)$ belongs to $G'_x=\psi(G_v)$. After replacing $g$ with $gh$ for some $h\in G_v$, we can assume that $n'=\psi(g)$. Note that $G_e\subseteq G_v$ and $G_e\subseteq G_w$ as the vertices $v,w$ are in the cylinder of $G_e$, so $gG_eg^{-1}$ and $G_e$ are contained in $G_w$, and by injectivity of $\psi$ on $G_w$ we must have $gG_eg^{-1}=G_e$, otherwise $n'G'_{e'}n'^{-1}\neq G'_{e'}$, contradicting the assumption that $n'$ belongs to $N'$. Therefore $g$ belongs to $N$, moreover we have $gG_vg^{-1}=G_w$, and thus $gN_vg^{-1}=N_w$, which is a contradiction. 

\smallskip \noindent Hence $\psi$ and $\psi'$ induce bijections between the conjugacy classes of the vertex groups of the cylinders $Y$ and $Y'$. Define $C=G_e$ and $C'=G'_{e'}$. Let $N_1,\ldots,N_r$ be the vertex groups of $Y/N$, and let $N'_1,\ldots,N'_r$ be the vertex groups of $Y'/N'$, numbered in such a way that $N_i$ is mapped to a conjugate of $N'_i$ by $\psi$, and vice versa. Since the underlying graphs of $Y/N$ and $Y'/N'$ are trees, we can write $N=N_1\ast_C N_2\ast_C \cdots \ast_C N_r$ and $N'=N'_1\ast_{C'} N'_2\ast_{C'} \cdots \ast_{C'} N'_r$. Without loss of generality, one can assume that $\psi(N_1)=N'_1$. There is an element $g\in N'$ such that $\psi(N_2)=gN'_2g^{-1}$. By assumption the image of the natural map $N'\rightarrow \mathrm{Aut}(C')$ is equal to $\mathrm{Inn}(C')$, therefore there exists an element $h\in C'$ such that $gxg^{-1}=hxh^{-1}$ for every $x\in C'$. Thus, one can define a new morphism $\theta : N \rightarrow N'$ as follows: $\theta_{\vert N_2}=\mathrm{ad}(hg^{-1})\circ \psi_{\vert N_2}$ and $\theta_{\vert N_i}=\psi_{\vert N_i}$ for $i\neq 2$. Note that $\theta(N_1)=N'_1$ and $\theta(N_2)=hN'_2h^{-1}=N'_2$. Define $\psi_2:=\theta$. Iteratively, one can define a morphism $\psi_r : N \rightarrow N'$ that coincides with $\psi$ up to conjugation on each vertex group $N_i$ of $N$, and such that $\psi_r(N_i)=N'_i$ for every $1\leq i\leq r$. Similarly, one can define a morphism $\psi_r' : N' \rightarrow N$ that coincides with $\psi'$ up to conjugation on each vertex group $N'_i$ of $N$, and such that $\psi_r'(N'_i)=N_i$ for every $1\leq i\leq r$. The endomorphism $\psi_r'\circ \psi_r$ of $N$ is surjective, and thus it is an automorphism since $N$ is Hopfian. Similarly $\psi_r\circ \psi_r'$ is an automorphism of $N'$. Hence $\psi_r$ and $\psi_r'$ are isomorphisms between $N$ and $N'$.

\smallskip \noindent Let $e_1,\ldots,e_n$ denote the edge groups of $\Delta$, and let $e'_1,\ldots,e'_n$ denote the edge groups of $\Delta'$ (the underlying graphs are trees and they have the same number of vertices, so they have the same number of edges as well), and let $C_1,\ldots,C_n$ and $C'_1,\ldots,C'_n$ denote the corresponding edge groups. By performing the same procedure as described above for each edge $e_i$ of $\Delta$, we get a collection of isomorphisms $\lbrace\chi_i : N_G(C_i)\rightarrow N_{G'}(C'_i)\rbrace_{1\leq i\leq n}$ (after renumbering the edges $e'_i$ if necessary) and $\lbrace\chi'_i : N_{G'}(C'_i)\rightarrow N_{G}(C_i)\rbrace_{1\leq i\leq n}$ that coincide with $\psi$ and $\psi'$ up to conjugacy on the vertex groups of the splittings of $N_G(C_i)$ and $N_{G'}(C'_i)$ respectively. 

\smallskip \noindent Let $T_c$ and $T'_c$ denote the trees of cylinders of $T$ and $T'$ respectively, and define $\Lambda=T_c/G$ and $\Lambda'=T'_c/G'$. Using the collection $\lbrace \chi_1,\ldots,\chi_n\rbrace$, one can define a morphism $\chi: G\rightarrow G'$ that coincides (up to conjugation) with $\chi_i$ on $N_{G}(G_{e_i})$ for every $1\leq i\leq n$ and with $\psi$ (up to conjugation) on the vertex groups of $\Delta$. Hence $\chi$ maps every vertex group of $\Lambda$ isomorphically to a vertex group of $\Lambda'$, and every edge group of $\Lambda$ isomorphically to an edge group of $\Lambda'$. One can define a similar morphism $\chi':G'\rightarrow G$. Now, the compositions $\chi'\circ \chi : G\rightarrow G$ and $\chi\circ \chi' : G'\rightarrow G'$ satisfy the assumptions of Lemma \ref{perin}, so $\chi'\circ \chi$ and $\chi\circ \chi'$ are isomorphisms. Therefore $G$ and $G'$ are isomorphic.\end{proof}

\color{black}In fact, in the proof of Theorem \ref{particular_case}, we proved the following result.


\begin{prop}Let $G$ and $G'$ be two torsion-generated hyperbolic groups. Let $\Delta$ and $\Delta'$ be reduced splittings of $G$ and $G'$ over finite groups, and suppose that all the edge groups of $\Delta$ have the same order. Moreover, suppose that the following condition holds: for every edge group $F$ of $\Delta$, the image of the natural map $N_G(F)\rightarrow \mathrm{Aut}(F)$ is equal to $\mathrm{Inn}(F)$.

\smallskip \noindent If there exist two morphisms $\varphi : G\rightarrow G'$ and $\varphi' : G'\rightarrow G$ that map each vertex group of $\Delta$ (respectively $\Delta'$) isomorphically to a vertex group of $\Delta'$ (respectively $\Delta$), that are injective on finite subgroups of $G$ (respectively $G'$), and that induce one-to-one correspondences between the conjugacy classes of vertex groups of $\Delta$ and $\Delta'$, and between the conjugacy classes of finite subgroups of $G$ and $G'$, then $G$ and $G'$ are isomorphic. More precisely, there exist two isomorphisms $\chi : G\rightarrow G'$ and $\chi' : G'\rightarrow G$ and an integer $m\geq 1$ such that:
\begin{enumerate}[(1)]
    \item[$\bullet$] for every vertex group $G_v$ of $\Delta$, there exists an element $g'\in G'$ such that $\chi$ coincides with $\mathrm{ad}(g')\circ \varphi\circ (\varphi'\circ \varphi)^{m-1}$ on $G_v$;
    \item[$\bullet$] for every vertex group $G'_v$ of $\Delta'$, there exists an element $g\in G$ such that $\chi'$ coincides with $\mathrm{ad}(g)\circ \varphi'\circ (\varphi\circ \varphi')^{m-1}$ on $G'_v$.
\end{enumerate}
\end{prop}

We need to prove that the previous proposition remains true without the assumption that all the edge groups of $\Delta$ have the same order. 


\begin{prop}\label{induction}Let $G$ and $G'$ be two torsion-generated hyperbolic groups. Let $\Delta$ and $\Delta'$ be reduced splittings of $G$ and $G'$ over finite groups. Moreover, suppose that the following condition holds: for every edge group $F$ of $\Delta$, the image of the natural map $N_G(F)\rightarrow \mathrm{Aut}(F)$ is equal to $\mathrm{Inn}(F)$.

\smallskip \noindent If there exist two morphisms $\varphi : G\rightarrow G'$ and $\varphi' : G'\rightarrow G$ that map each vertex group of $\Delta$ (respectively $\Delta'$) isomorphically to a vertex group of $\Delta'$ (respectively $\Delta$), that are injective on finite subgroups of $G$ (respectively $G'$), and that induce one-to-one correspondences between the conjugacy classes of vertex groups of $\Delta$ and $\Delta'$, and between the conjugacy classes of finite subgroups of $G$ and $G'$, then $G$ and $G'$ are isomorphic. More precisely, there exist two isomorphisms $\chi : G\rightarrow G'$ and $\chi' : G'\rightarrow G$ and an integer $m\geq 1$ such that:
\begin{enumerate}[(1)]
    \item[$\bullet$] for every vertex group $G_v$ of $\Delta$, there exists an element $g'\in G'$ such that $\chi$ coincides with $\mathrm{ad}(g')\circ \varphi\circ (\varphi'\circ \varphi)^{m-1}$ on $G_v$;
    \item[$\bullet$] for every vertex group $G'_v$ of $\Delta'$, there exists an element $g\in G$ such that $\chi'$ coincides with $\mathrm{ad}(g)\circ \varphi'\circ (\varphi\circ \varphi')^{m-1}$ on $G'_v$.
\end{enumerate}
\end{prop}

\begin{proof}
We denote by $N(\Delta,\Delta')$ the sum of the numbers of edges in $\Delta$ and $\Delta'$. We will prove the proposition by induction on $N(\Delta,\Delta')$ for the tuple $(G,G',\Delta,\Delta',\varphi,\varphi')$.

\smallskip \noindent When $N(\Delta,\Delta')=0$ the result follows immediately from the assumption. So let $n\geq 0$ be an integer and suppose that the proposition is true for every $\Delta,\Delta'$ such that $N(\Delta,\Delta')\leq n$. Let $\Delta,\Delta'$ be such that $N(\Delta,\Delta')=n+1$. If all the edge groups of $\Delta$ have the same order, then the result is true by Proposition \ref{particular_case}, so let us suppose that there are edge groups of $\Delta$ of different order. Let $m$ be the minimal order of an edge group of $\Delta$ or $\Delta'$, and let $\Lambda$ and $\Lambda'$ be the splittings of $G$ and $G'$ obtained by collapsing all the edges whose stabilizer has order $> m$. Note that $N(\Lambda,\Lambda')\leq n$.

\smallskip \noindent Since $\varphi$ and $\varphi'$ are injective on the finite subgroups of $G$ and $G'$, the following holds: for every vertex group $G_v$ of $\Lambda$, $\varphi(G_v)$ is elliptic in $\Lambda'$, and for every vertex group $G'_v$ of $\Lambda'$, $\varphi'(G'_v)$ is elliptic in $\Lambda$. After replacing $\varphi$ and $\varphi'$ by $\varphi\circ (\varphi'\circ \varphi)^{k-1}$ and $\varphi'\circ (\varphi\circ \varphi')^{k-1}$ for some integer $k\geq 1$ if necessary, one can assume that for every vertex group $H=G_v$ of $\Lambda$, $\varphi(H)$ is contained in $H'=G'_{v'}$ and that $\varphi'(H')$ is contained in (a conjugate of) $H$. Let $\Delta_H$ and $\Delta'_{H'}$ be the reduced splittings of $H$ and $H'$ over finite groups coming from $\Delta$ and $\Delta'$. The induction hypothesis can be applied to $(H,H',\Delta_H,\Delta'_{H'},\varphi_{\vert H},\varphi'_{\vert H'})$ and provides a pair of isomorphisms $\psi_H : H\rightarrow H'$ and $\psi'_{H'} : H'\rightarrow H$ such that $\psi_H$ coincides up to conjugacy with $ \varphi\circ (\varphi'\circ \varphi)^{m-1}$ on $H$ and $\psi'_{H'}$ coincides up to conjugacy with $ \varphi'\circ (\varphi\circ \varphi')^{m-1}$ on $H'$ for some integer $m\geq 1$. Applying the induction hypothesis to each pair of vertex groups of $\Lambda$ and $\Lambda'$, we get a pair of morphisms $\theta : G\rightarrow G'$ and $\theta' : G'\rightarrow G$, and we conclude by applying the induction hypothesis to $(G,G',\Lambda,\Lambda',\theta,\theta')$.\end{proof}

It remains to prove the general case of Theorem \ref{main_theorem2}. 

\begin{proof}[Proof of Theorem \ref{main_theorem2}]
Let $G$ be a torsion-generated hyperbolic group, and let $G'$ be a finitely torsion-generated group that is AE-equivalent to $G$. Let $\Delta$ and $\Delta'$ be reduced Stallings splittings of $G$ and $G'$ respectively. By Corollary \ref{main_theorem1.1}, there exist two morphisms $\varphi : G\rightarrow G'$ and $\varphi' : G'\rightarrow G$ that are injective on the vertex groups of $\Delta$ and $\Delta'$, and that induce one-to-one correspondences between the conjugacy classes of vertex groups of $\Delta$ and $\Delta'$, and between the conjugacy classes of finite subgroups of $G$ and $G'$. It follows that $G'$ is a hyperbolic group (by the same argument as the one at the end of the proof of Corollary \ref{preservation}). Finally, the result follows from Proposition \ref{induction}.\end{proof}\color{black}

\section{A counterexample to first-order torsion-rigidity}\label{example}

In this section, we give an example of two virtually free Coxeter groups $G$ and $G'$ that are AE-equivalent, but which are not isomorphic.

\subsection{Definition of the groups $G$ and $G'$}\label{def_of_groups}

Consider the finite Coxeter groups $A,B,C$ given by the following diagrams (let us recall that the vertices represent the elements of a Coxeter system $S$, and that there is no edge between two vertices if and only if the corresponding generators $s_i,s_j$ commute, and an edge with no label if and only if $(s_is_j)^3=1$):

\begin{figure}[h!]
  \centering
    \begin{tikzpicture}
\draw[black, very thick] (-2,0) -- (-1,0);
\draw[black, very thick] (0,0) -- (5,0);
\draw[black, very thick] (6,0) -- (8,0);
\node[text=black] at (-2,-0.5) {$a_1$};
\node[text=black] at (0,-0.5) {$a_2$};
\node[text=black] at (1,-0.5) {$x$};
\node[text=black] at (2,-0.5) {$a_3$};
\node[text=black] at (4,-0.5) {$a_4$};
\node[text=black] at (6,-0.5) {$a_5$};
\fill[black] (-2,0) circle (0.1cm);
\fill[black] (-1,0) circle (0.1cm);
\fill[black] (0,0) circle (0.1cm);
\fill[black] (1,0) circle (0.1cm);
\fill[black] (2,0) circle (0.1cm);
\fill[black] (3,0) circle (0.1cm);
\fill[black] (4,0) circle (0.1cm);
\fill[black] (5,0) circle (0.1cm);
\fill[black] (6,0) circle (0.1cm);
\fill[black] (7,0) circle (0.1cm);
\fill[black] (8,0) circle (0.1cm);
\end{tikzpicture}
  \caption{The finite group $A$, isomorphic to $S_3\times S_7\times S_4$.}
\end{figure}
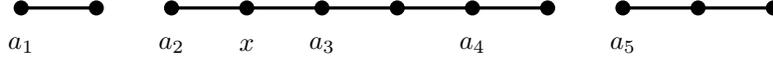

\begin{figure}[h!]
  \centering
    \begin{tikzpicture}
\draw[black,very thick] (-1,0) -- (4,0);
\draw[black,very thick] (5,0) -- (7,0);
\fill[black] (-1,0) circle (0.1cm);
\fill[black] (0,0) circle (0.1cm);
\fill[black] (1,0) circle (0.1cm);
\fill[black] (2,0) circle (0.1cm);
\fill[black] (3,0) circle (0.1cm);
\fill[black] (7,0) circle (0.1cm);
\fill[black] (4,0) circle (0.1cm);
\fill[black] (5,0) circle (0.1cm);
\fill[black] (6,0) circle (0.1cm);
\node[text=black] at (-1,-0.5) {$b_1$};
\node[text=black] at (1,-0.5) {$b_2$};
\node[text=black] at (3,-0.5) {$b_4$};
\node[text=black] at (5,-0.5) {$b_3$};
\node[text=black] at (7,-0.5) {$b_5$};
\end{tikzpicture}
  \caption{The finite group $B$, isomorphic to $S_7\times S_4$.}
\end{figure}
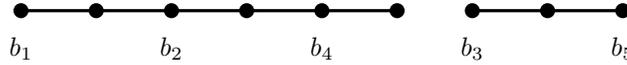

\begin{figure}[h!]
  \centering
    \begin{tikzpicture}
\fill[black] (0,0) circle (0.1cm);
\fill[black] (1,0) circle (0.1cm);
\fill[black] (2,0) circle (0.1cm);
\fill[black] (3,0) circle (0.1cm);
\fill[black] (4,0) circle (0.1cm);
\node[text=black] at (0,-0.5) {$c_1$};
\node[text=black] at (1,-0.5) {$c_2$};
\node[text=black] at (2,-0.5) {$c_3$};
\node[text=black] at (3,-0.5) {$c_4$};
\node[text=black] at (4,-0.5) {$c_5$};
\end{tikzpicture}
  \caption{The finite group $C$, isomorphic to $(\mathbb{Z}/2\mathbb{Z})^5$.}
\end{figure}
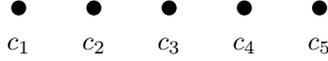

We will need the following easy lemmas.

\begin{lemma}\label{lemma_autom_inner}
$\mathrm{Aut}(A)=\mathrm{Inn}(A)$ and $\mathrm{Aut}(B)=\mathrm{Inn}(B)$.
\end{lemma}

\begin{proof}
The groups $S_3,S_7,S_4$ have trivial center and do not have common direct factors, so $\mathrm{Aut}(A)\simeq \mathrm{Aut}(S_3)\times \mathrm{Aut}(S_7)\times \mathrm{Aut}(S_4)$ (see for instance \cite{Bidwell,Bidwell2}). Moreover, every automorphism of $S_n$ is inner for $n\neq 6$, hence every automorphism of $A$ is inner.


\smallskip \noindent Similarly, every automorphism of $B$ is inner.
\end{proof}

Let $C_A$ be the subgroup of $A$ generated by $\lbrace a_1,a_2,a_3,a_4,a_5\rbrace$ and let $C_B$ be the subgroup of $B$ generated by $\lbrace b_1,b_2,b_3,b_4,b_5\rbrace$. Note that $C_A\simeq C_B\simeq C\simeq (\mathbb{Z}/2\mathbb{Z})^5$.

\begin{lemma}\label{lemma_A}
There is a natural morphism \[N_A(C_A)\rightarrow S(\lbrace a_1,a_2,a_3,a_4,a_5\rbrace)\] whose image is $S(\lbrace a_2,a_3,a_4\rbrace)$.
\end{lemma}

\begin{proof}
Let $g$ be an element of $N_A(C_A)$. As the inner automorphism $\mathrm{ad}(g)$ of $A$ preserves each direct factor of the decomposition of $A$ as a direct product, we have $ga_1g^{-1}=a_1$, $ga_5g^{-1}=a_5$ and $g\langle a_2,a_3,a_4\rangle g^{-1}=\langle a_2,a_3,a_4\rangle$. Moreover, since a transposition of $S_n$ is not conjugate to a product of two transpositions, we have $g\lbrace a_2,a_3,a_4\rbrace g^{-1}=\lbrace a_2,a_3,a_4\rbrace$. Then, one easily checks that the involution $a=xa_2a_3x\in A$ satisfies $aa_2a^{-1}=a_3$ (and so $aa_3a^{-1}=a_2$). Similarly, there is an involution swapping $a_3$ and $a_4$. Hence the image is the full group $S(\lbrace a_2,a_3,a_4\rbrace)$.
\end{proof}


One can prove the following fact in the same way.

\begin{lemma}\label{lemma_B}
There is a natural morphism \[N_B(C_B)\rightarrow S(\lbrace b_1,b_2,b_3,b_4\rbrace)\] whose image is $S(\lbrace b_1,b_2,b_4\rbrace)\times S(\lbrace b_3,b_5\rbrace)$.
\end{lemma}

Now, define an embedding $i: C\rightarrow C_A\subseteq A$ by $i(c_n)=a_n$ for every $1\leq n\leq 5$, and two embeddings $j : C\rightarrow C_B\subseteq B$ by $j(c_n)=b_n$ for every $1\leq n\leq 5$, and $k : C\rightarrow C_B\subseteq B$ by $k(c_1)=b_3$, $k(c_2)=b_1$, $k(c_3)=b_5$, $k(c_4)=b_4$ and $k(c_5)=b_2$. 

Then, define two amalgamated products as follows:\[G=\langle A,B \ \vert \ i(c)=j(c), \ \forall c\in C\rangle,\]\[G'=\langle A,B \ \vert \ i(c)=k(c), \ \forall c\in C\rangle.\]

One easily sees that $G$ and $G'$ are Coxeter groups. Furthermore, $G$ and $G'$ are virtually free (and non-elementary) since $A$ and $B$ are finite. We will prove that $G$ and $G'$ are not isomorphic, but that they are AE-equivalent. From now on, let us forget $i$ and $j$, so that $C_A,C,C_B$ are identified via $a_n=c_n=b_n$ for every $1\leq n\leq 5$.

\subsection{The groups $G$ and $G'$ are not isomorphic}

\begin{lemma}Consider two amalgamated products $G_1=A_1\ast_{C_1}B_1$ and $G_2=A_2\ast_{C_2}B_2$, where the groups $A_1,A_2,B_1,B_2$ are finite and $\vert A_i\vert>\vert C_i\vert$ and $\vert B_i\vert >\vert C_i\vert$ for $i\in\lbrace 1,2\rbrace$. Suppose that $A_1$ and $B_1$ are not isomorphic, and that $A_2$ and $B_2$ are not isomorphic. If $G_1$ and $G_2$ are isomorphic, then there exists an isomorphism $\varphi : G_1\rightarrow G_2$ such that $\varphi(A_1)=A_2$ and $\varphi(B_1)=B_2$.\end{lemma}

\begin{proof}Let $\varphi : G_1\rightarrow G_2$ be an isomorphism. Note that the groups $\varphi(A_1)$ and $\varphi(B_1)$ are elliptic in the Bass-Serre tree of the splitting $G_2=A_2\ast_{C_2}B_2$. It follows that
$\varphi(A_1)$ and $\varphi(B_1)$ are conjugate to $A_2$ and $B_2$ respectively (because $A_i$ and $B_i$ are not isomorphic, for $i\in\lbrace 1,2\rbrace$). Thus we can assume, after composing $\varphi$ with an inner automorphism of $G_2$ if necessary, that $\varphi(A_1)=A_2$. We will prove that $\varphi(B_1)=aB_2a^{-1}$ for some element $a\in A_2$, and therefore that $\mathrm{ad}(a^{-1})\circ \varphi(A_1)=A_2$ and $\mathrm{ad}(a^{-1})\circ \varphi(B_1)=B_2$.

\smallskip \noindent In what follows, we denote by $(T_1,d_1)$ the  Bass-Serre tree of the splitting $A_1\ast_{C_1} B_1$ of $G_1$ equipped with the simplicial metric, and by $(T_2,d_2)$ the Bass-Serre tree of the splitting $A_2\ast_{C_2} B_2$ of $G_2$ equipped with the simplicial metric.

\smallskip \noindent We will prove that $\varphi$ induces an isometry between $(T_1,d_1)$ and $(T_2,d_2)$. We could use the fact that the deformation spaces $\mathcal{D}(T_1)$ and $\mathcal{D}(T_2)$ are reduced to a point, but instead we will give a direct construction of the isometry induced by $\varphi$.


\smallskip \noindent First, let us define a $\varphi$-equivariant map $f$ from $T_1$ to $T_2$ as follows: let $v$ and $w$ be two adjacent vertices of $T_1$ fixed respectively by $A_1$ and $B_1$, and let $v'$ and $w'$ be two adjacent vertices of $T_2$ fixed respectively by $A_2$ and $B_2$. Note that $v$ is the unique vertex of $T_1$ fixed by $A_1$ (as the order of $A_1$ is strictly greater than the order of the edges of $T_1$), and that $w$ is the unique vertex of $T_1$ fixed by $B_1$ (for the same reason). Similarly, $v'$ is the unique vertex of $T_2$ fixed by $\varphi(A_1)=A_2$, and $w'$ is the unique vertex of $T_2$ fixed by $B_2$. Let $g\in G_2$ be an element such that $\varphi(B_1)=gB_2g^{-1}$, and note that $\varphi(B_1)$ fixes $gw'$. We define $f(v)=v'$ and $f(w)=gw'$. Then, since every vertex of $T_1$ is in the orbit of $v$ or $w$, we define $f$ on each vertex of $T_1$ by $\varphi$-equivariance (if $u=hv$ we set $f(u)=\varphi(h)f(v)$ and if $u=hw$ we set $f(u)=\varphi(h)f(w)$). It remains to define $f$ on the edges of $T_1$: if $e$ is an edge of $T_1$, with endpoints $v_1$ and $v_2$, then, since $T_2$ is a tree, there exists a unique path $e'$ from $f(v_1)$ to $f(v_2)$ in $T_2$, and we define $f(e) = e'$. 

\smallskip \noindent Note that $f$ maps the edge $[v,w]$ to the path $[v',gw']$, which is not a vertex since $v'\neq gw'$ (as $v'$ and $w'$ are not in the same $G_2$-orbit). Therefore, as every edge of $T_1$ is a $G_1$-translate of $[v,w]$, the map $f$ cannot map an edge of $T_1$ to a vertex of $T_2$. Thus, after subdividing the edges in $T_1$ if necessary, one can assume that $f$ maps every edge of $T_1$ to an edge of $T_2$. Let $d'_1$ denote the new metric on $T_1$ obtained after subdividing the edges of $T_1$ (if necessary). Let us prove that $f$ is an isometry between $(T_1,d'_1)$ and $(T_2,d_2)$. Since $f$ does not map any edge to a vertex, it suffices to prove that there is no folding, i.e., that we cannot find two distinct edges $e,e'$ with a common vertex such that $f(e)=f(e')$. So let us suppose towards a contradiction that $f$ folds two distinct edges $e$ and $e'$ with a common vertex. One can assume, after translating the edges if necessary, that $e=[v,u]$ and $e'=[v,au]$ for some $a\in A_1$, where $u$ denotes a vertex of the path $[v,w]$ adjacent to $v$ (or $e=[w,u]$ and $e'=[w,bu]$ for some $b\in B_1$, where $u$ denotes a vertex of the path $[w,v]$ adjacent to $w$). Let $C\subseteq A_1$ denote the stabilizer of the edge $e$. We have $e'=ae$ and therefore $f(e')=\varphi(a)f(e)=f(e)$, thus $\varphi(a)$ normalizes $\varphi(C)$ and by injectivity of $\varphi$ on $A_1$ the element belongs to the normalizer of $C$ in $A_1$, contradicting the assumption that $e'\neq e$. Hence $f$ is an isometry from $(T_1,d'_1)$ to $(T_2,d_2)$.

\smallskip \noindent Then, let us prove that $d'_1=d_1$, or equivalently that the vertex $f(w)=gw'$ is at distance one from the vertex $f(v)=v'$. Let $x$ be a vertex on the path $[f(v),f(w)]$ that is adjacent to $f(v)=v'$. There is an element $g'\in (G_2)_{f(v)}=A_2$ such that $x=g'w'=g'g^{-1}f(w)$. Since $\varphi$ is assumed to be surjective, one can write $g'g^{-1}=\varphi(h)$ for some $h\in G_1$. Hence $x=\varphi(h)f(w)=f(hw)$. Now, since the path $[v,hw]$ is mapped to the edge $[v',x]$ in $T_2$ by the isometry $f$, the path $[v,hw]$ is an edge of $T_1$. But there is no folding, so $hw=w$ and thus $x=f(w)$, which proves that $[f(v),f(w)]$ is an edge of $T_2$. But $f(v)=v'$, so there exists an element $a\in A_2$ such that $f(w)=aw'$. Therefore, $\mathrm{ad}(a^{-1})\circ \varphi(A_1)=A_2$, $\mathrm{ad}(a^{-1})\circ \varphi(B_1)=B_2$ and $\mathrm{ad}(a^{-1})\circ\varphi(C_1)=C_2$.\end{proof}

We are ready to prove that $G$ and $G'$ are not isomorphic. So let us suppose towards a contradiction that there exists an isomorphism $\varphi : G \rightarrow G'$. Then, by the previous lemma, one can assume that $\varphi(A)=A$ and that $\varphi(B)=B$. It follows that $\varphi(A\cap B)=A\cap B$ (indeed, $\varphi(A\cap B)$ is contained in $\varphi(A)\cap \varphi(B)=A\cap B$, and both $\varphi(A\cap B)$ and $A\cap B$ have order $\vert C\vert$). Thus $\varphi(C_A)=C_A$ (and $C_A=C_B$ in $G$ by definition).

Since $\mathrm{Aut}(A)=\mathrm{Inn}(A)$ and $\mathrm{Aut}(B)=\mathrm{Inn}(B)$ (by Lemma \ref{lemma_autom_inner}), there exists an element $a\in A$ and an element $b\in B$ such that $\varphi_{\vert A}=\mathrm{ad}(a)$ and $\varphi_{\vert B}=\mathrm{ad}(b)$. Moreover, these elements belong to $N_A(C_A)$ and $N_B(C_B)$ respectively since $\varphi(C_A)=C_A$ and $\varphi(C_B)=C_B$. 

Then, let $\alpha, \beta$ denote the permutations of $\lbrace 1,2,3,4,5\rbrace$ such that $aa_ia^{-1}=a_{\alpha(i)}$ and $bb_ib^{-1}=b_{\beta(i)}$, and let $\kappa$ denote the permutation of $\lbrace 1,2,3,4,5\rbrace$ such that $k(c_i)=b_{\kappa(i)}$, namely $\kappa=(2 \ 1 \ 3 \ 5)$.

The relation $a_i=b_i$ in $G$ is mapped to the relation $\varphi(a_i)=\varphi(b_i)$ in $G'$, that is $a_{\alpha(i)}=b_{\beta(i)}$. But the relation $a_i=b_{\kappa(i)}$ holds in $G'$, therefore we get $\kappa\alpha=\beta$, so $\kappa=\beta\alpha^{-1}$. But this is not possible, because $\alpha$ fixes $5$ (by Lemma \ref{lemma_A}) and $\beta(5)$ belongs to $\lbrace 3,5\rbrace$ (by Lemma \ref{lemma_B}), thus $\kappa(5)$ cannot be equal to $2$.

\subsection{The groups $G$ and $G'$ are AE-equivalent}

The following definition was introduced in \cite{And19a}.

\begin{de}\label{legal}
Let $G$ be a hyperbolic group. Suppose that $G$ is non-elementary (i.e., $G$ is infinite and not virtually $\mathbb{Z}$). Let $C$ be a finite subgroup of $G$. The group $G\ast_C=\langle G,t \ \vert \ tc=ct, \ \forall c\in C\rangle$ is called a \emph{legal large extension of $G$} if $N_G(C)$ is non-elementary and $E_G(N_G(C))=C$, where $E_G(N_G(C))$ denotes the unique maximal finite subgroup of $G$ normalized by $N_G(C)$ (which always exists in a hyperbolic group).\end{de}

We will use the following theorem, proved in \cite{And19a}, to show that the groups $G$ and $G'$ defined in Subsection \ref{def_of_groups} are AE-equivalent.

\begin{te}\label{theorem_legal}
Let $G$ and $G'$ be non-elementary hyperbolic groups. If there exist legal large extensions $\Gamma$ and $\Gamma'$ of $G$ and $G'$ respectively such that $\Gamma\simeq \Gamma'$, then $G$ and $G'$ are AE-equivalent.
\end{te}

We will define two groups $\Gamma$ and $\Gamma'$ as in the theorem above. 

Define $\alpha = (2 \ 3 \ 4)$ and $\beta = (3 \ 5)(1 \ 2 \ 4)$. One easily checks that $\kappa \alpha \kappa^{-1}=\beta\kappa$. By Lemmas \ref{lemma_A} and \ref{lemma_B}, there exist two elements $a\in A$ and $b\in B$ such that, for every $1\leq i\leq 5$: \[aa_ia^{-1}=a_{\alpha(i)} \ \ \ \text{and} \ \ \  bb_ib^{-1}=b_{\beta(i)}.\]Then, define $\varphi : G \rightarrow G'$ by $\varphi_{\vert A}=\mathrm{ad}(b)$ and $\varphi_{\vert B}=\mathrm{ad}(a)$. This morphism is well-defined because the relation $a_i=b_i$, which holds in $G$, is mapped to $ba_ib^{-1}=ab_ia^{-1}$ in $G'$, that is $bb_{\kappa(i)}b^{-1}=aa_{\kappa^{-1}(i)}a^{-1}$, i.e., $b_{\beta\kappa(i)}=b_{\kappa\alpha\kappa^{-1}(i)}$. This latter relation holds in $G'$ by our choice of $\alpha$ and $\beta$. Note that the image of $\varphi$ is the subgroup of $G'$ generated by $bAb^{-1}$ and $aBa^{-1}$.

Moreover, note that $\kappa=(2 \ 3)(2 \ 1)(3 \ 5)$. Define $\alpha'=(2 \ 3)$ and $\beta'=(2 \ 1)(3 \ 5)$. Again by Lemmas \ref{lemma_A} and \ref{lemma_B}, there exist $a'\in A$ and $b'\in B$ such that, for every $1\leq i\leq 5$: \[a'a_ia'^{-1}=a_{\alpha'(i)} \ \ \ \text{and} \ \ \ b'b_ib'^{-1}=b_{\beta'(i)}.\] Finally, define $\varphi' : G'\rightarrow G$ by $\varphi'_{\vert A}=\mathrm{ad}(b')$ and $\varphi'_{\vert B}=\mathrm{ad}(a')$. This morphism is well-defined because the relation $a_i=b_{\kappa(i)}$, which holds in $G'$, is mapped to $b'a_ib'^{-1}=a'b_{\kappa(i)}a'^{-1}$, that is $b_{\beta'(i)}=b_{\alpha'\kappa(i)}$. This relation holds in $G$ by our choice of $\alpha'$ and $\beta'$. Note that the image of $\varphi'$ is the subgroup of $G$ generated by $b'Ab'^{-1}$ and $a'Ba'^{-1}$.

Define \[\Gamma=\langle G, t \ \vert \ tc=ct \ \forall c\in C\rangle\] and \[\Gamma'=\langle G', t' \ \vert \ t'c=ct' \ \forall c\in C\rangle,\] and a morphism $\psi : \Gamma\rightarrow \Gamma'$ by $\psi= \mathrm{ad}(b^{-1})\circ \varphi=\mathrm{id}$ on $A$, $\psi=\mathrm{ad}(t'b^{-1}a^{-1})\circ \varphi=\mathrm{ad}(t'b^{-1})$ on $B$ and $\psi(t)=t'$. This morphism is well-defined because $t'$ centralizes $C$. Similarly, let us define a morphism $\psi' : \Gamma'\rightarrow \Gamma$ by $\psi'= \mathrm{ad}(b'^{-1})\circ \varphi'=\mathrm{id}$ on $A$, $\psi'=\mathrm{ad}(tb'^{-1}a'^{-1})\circ \varphi'=\mathrm{ad}(tb'^{-1})$ on $B$ and $\psi'(t')=t$. 

Note that $\psi$ is surjective since its image contains $A$, the stable letter $t'$ and $\psi(B)=t'Bt'^{-1}$, and similarly $\psi'$ is surjective. But virtually free groups (and more generally hyperbolic groups) are Hopfian, so $\psi\circ \psi'$ is an automorphism of $\Gamma'$ and therefore $\psi$ and $\psi'$ are isomorphisms.

Moreover, the assumptions of Definition \ref{legal} are clearly satisfied by $\Gamma$ and $\Gamma'$ since $C$ is normal in $G$ and $G'$. Therefore, by Theorem \ref{theorem_legal}, $G$ and $G'$ are AE-equivalent.

\section{Domains}
We follow \cite{KVASCHUK200578}. In particular, we expand on \cite[Remark 1, pg. 92]{KVASCHUK200578} where we replace abelian by abelian-by-finite. We invite the reader to refer to \cite{KVASCHUK200578} for details.
	
	\begin{de} Let $G$ be a group. For $x, y \in G$ we define the following operation:
	$$x \diamond y = [\mathrm{gp}_G(x), \mathrm{gp}_G(y)],$$
where $\mathrm{gp}_G(x)$ denotes the normal closure of $x$ in $G$, i.e., the smallest normal subgroup of $G$ containing $x$. We call a non-trivial element $x \in G$ a zero-divisor in $G$ if there exists a non-trivial element
$y \in G$ such that $x \diamond y = 1$. We say that the group $G$ is a domain if it has no zero-divisors. Finally, we write $x \perp y$ when $x \diamond y = 1$.
\end{de} 

The following lemma follows from standard techniques.

\begin{lemma}\label{standard}Every non-elementary hyperbolic group without a non-trivial normal finite subgroup is a domain. 
\end{lemma}

	\begin{notation} 
	$\mathrm{Comp}(x,z) = \forall y (y \diamond z = 1 \rightarrow x \diamond y = 1)$ (cf. \cite[Proof of Lemma~4]{KVASCHUK200578}).
\end{notation}

		\begin{notation}\label{Dk_notation} As in \cite{KVASCHUK200578}, we denote by $D_k$ the groups of the form $G_1 \times \cdots \times G_k$ with each $G_i$ a domain.
\end{notation}

		\begin{te}\label{the_domain_ab_by_finite_th} Let $H = H_1 \times H_2$ with $H_2 \in D_k$ and $H_1$ abelian-by-finite. Then for every $K \equiv H$ there exist $K_1, K_2 \leq K$ such that the following hold:
	\begin{enumerate}[(1)]
	\item $K = K_1 \times K_2$;
	\item $K_2 \in D_k$;
	\item $K_1 \equiv H_1$ and $K_2 \equiv H_2$. 
\end{enumerate}	
	\end{te}
	
	\begin{proof} Let $A, C \leq H$ be such that $H_1 = AC$ for $A$ abelian and normal in $H_1$, $A$ maximal with respect to these properties, and $C$ finite. Let us consider the first-order sentence $\psi$ that says that there are $a_2, ..., a_{k+1}$ and $\bar{c} \in H^{< \omega}$ such that:
	\begin{enumerate}[(i)]
	\item $C' = \{c : c \in \bar{c} \}$ is a finite subgroup of $H$ isomorphic to $C$;
	\item $\mathrm{Aff}(H) = \{g \in H : g \perp g \}C'$ is a subgroup of $H$;
	\item for every $i \in [2, k+1]$, $\mathrm{Comp}(H, a_i)$ is a domain;
	\item $H = \mathrm{Aff}(H) \times \mathrm{Comp}(H, a_2) \times \cdots \times \mathrm{Comp}(H, a_{k+1})$.
\end{enumerate}	
We claim that $\mathrm{Aff}(H) = H_1$ for $C'=C$, so that $H \models \psi$, by an appropriate choice of elements $a_2, ..., a_{k+1} \in H_2$, in fact then, by an adaptation of \cite[Proposition~1(3)]{KVASCHUK200578}, we have that $H_2$ is the direct product of the $\mathrm{Comp}(H, a_i)$ for $i \in [2, k+1]$. Let $\pi_2 : H \rightarrow H_2$ be the canonical projection onto $H_2$ and suppose that there is $g \in H$ such that $H \models g \perp g$ and $e \neq \pi_2(g)$, then there is $h \in H_2$ such that $H_2 \models h \perp h$, which contradicts the fact that $H_2 \in D_k$. Hence, $\{g \in H : g \perp g \} \subseteq H_1$. On the other hand we have that $A \subseteq \{g \in H : g \perp g \}$, and so we are done.
Now let $K \equiv H$. Then $K \models \psi$, and hence $K = K_1 \times K_2$, where $K_1 = \mathrm{Aff}(K) \equiv H_1$ and $K_2 = \mathrm{Comp}(K, b_2) \times \cdots \times \mathrm{Comp}(K, b_{k+1})$ for appropriate witnesses $b_2, ..., b_{k+1} \in K$. To conclude, it suffices to observe that our assumptions imply that
\[H/H_1 \equiv K/K_1, \; K/K_1 = K_2 \text{ and } H/H_1 = H_2.\]\end{proof}

	\begin{proof}[Proof of \ref{final_main_theorem}] Let $W$ be a Coxeter group and suppose that $W$ has an irreducible decomposition of the form $W_1 \times \cdots \times W_n$ with each component either finite (spherical), affine, or infinite and hyperbolic. Without loss of generality, there is $m < \omega$ such that $W_i$ is infinite hyperbolic if and only if $m < i \leq n$. Let $k = n - m$. It is possible that $k = 0$, in which case there are no infinite hyperbolic components, or that $k = n$, in which case there are no spherical or affine components. Then, letting $W^*_1 = \{W_i : 1 \leq i \leq m \}$ and $W^*_2 = \{W_i : m < i \leq n\}$, we have $W = W^*_1 \times W^*_2$. By convention, if $k = 0$, then $W^*_2$ is the trivial subgroup and it can be ignored, and if $k = n$, then $W^*_1$ is the trivial subgroup and it can be ignored.
 The group $W^*_1$ is abelian-by-finite, and by Lemma \ref{standard}, $W^*_2 \in D_k$ (recall \ref{Dk_notation}). Now let $G \equiv W$ and suppose that $G$ is finitely torsion-generated. By Theorem \ref{the_domain_ab_by_finite_th}, we have $G = G_1 \times G_2$ with $G_2 \in D_k$, $G_1 \equiv W^*_1$, and $G_2 \equiv W^*_2$. By \cite[Corollary~1.8]{MUHLHERR2022297} and Theorem~\ref{preservation}, $G_2$ is a Coxeter group, since $G_2$ is a direct product of $k$ torsion-generated groups. Finally, $G_1$ is abelian-by-finite, as $G_1 \equiv W^*_1$. By \cite[Theorem~2.2]{LASSERRE2014213} and the first-order rigidity of irreducible affine Coxeter groups proved in \cite{PS23}, we have $G_1 \simeq W^*_1$.\end{proof} 

    \begin{proof}[Proof of \ref{reduction_to_dragon}] Argue as in the proof of \ref{final_main_theorem}.
    \end{proof}

\section{Rigidity among even Coxeter groups}

    In this final section we prove Theorem~\ref{rigidity_Coxeter+}.

\begin{fact}\label{finite_subgroups_fact} Let $(W, S)$ be a Coxeter system. If $A \leq W$ is finite, then $A$ is contained in a finite special parabolic subgroup of $W$, i.e., there is $w \in W$ and $T \subseteq S$ such that $W_T = \langle T \rangle$ is finite and $A \leq wW_T w^{-1}$.
\end{fact}


\begin{fact}[{\cite[Section~4]{caprace_conj_separability}}]\label{proj_fact}  Let $(W, S)$ be an even Coxeter system. For every $I \subseteq S$ there is a canonical retraction $\pi_I \in \mathrm{End}(W)$ of $W$ onto the
standard parabolic subgroup $W_I$, defined by $\pi_I(s) = s$ for all $s \in I$ and
$\pi_I(t) = e$ for all $t \in S \setminus I$. Furthermore, for any other subset $J \subseteq S$, the retractions $\pi_I$ and $\pi_J$ commute.
\end{fact}

	\begin{de} Assume that $A$ and $B$ are two retracts of a group $G$ and $\pi_A, \pi_B \in \mathrm{End}(G)$ are the corresponding retractions. We will say $\pi_A$ commutes with $\pi_B$ if they
commute as elements of the monoid of endomorphisms $\mathrm{End}(G)$, that is:
$$\pi_A(\pi_B(g)) = \pi_B(\pi_A(g)), \text{ for all } g \in G.$$
\end{de} 

	\begin{fact}[{\cite[2.4]{caprace_conj_separability}}]\label{commuting_retraction}  If the retractions $\pi_A$ and $\pi_B$ commute, then $\pi_A(B) = A \cap B = \pi_B(A)$ and the endomorphism $\pi_{A\cap B} := \pi_A \circ \pi_B = \pi_B \circ
\pi_A$ \mbox{is a retraction onto $A \cap B$.}
\end{fact}

\begin{de}\label{special_monoid} Let $(W, S)$ be a Coxeter system. We say that $\alpha \in End(W)$ is  an $S$-self-similarity (or a special $S$-endomorphism) if for every $s, t \in S$ we have:
	\begin{enumerate}[(1)]
	\item $\alpha(s) \in s^W$;
	\item $o(\alpha(s) \alpha(t)) = o(st)$.
\end{enumerate}
\end{de}

	\begin{fact}[{\cite[3.23]{MUHLHERR2022297}}]\label{the_even_isomorphism_fact} Let $(W,S)$ be an even Coxeter system of finite rank, and let $\alpha$ be a self-similarity of $(W,S)$ (cf. Definition~\ref{special_monoid}). Then $(\langle \alpha(S) \rangle_W, \alpha(S))$ is a Coxeter system and thus the map $\alpha: W \rightarrow \langle \alpha(S) \rangle_W$ is an isomorphism.
\end{fact}

	\begin{fact}[{\cite[3.2.16]{davis}}]\label{deletion_condition} Let $(W, S)$ be a Coxeter system and denote with $\ell$ the corresponding length function. If $w = s_1s_2 \cdots s_k$ is a word in the alphabet $S$ and $\ell(w) <
k$, then $w = s_1 \cdots \hat{s}_i \cdots \hat{s}_j \cdots s_k$ for some $1 \leq i < j \leq k$, where, for $m \in [1, k]$, $\hat{s}_m$ denotes omission of the letter $s_m$ in the word  
$w = s_1s_2 \cdots s_k$.
\end{fact}

	\begin{fact}\label{hopfian} As the geometric representation of $(W,S)$ is faithful (see e.g. \cite{Humphreys_1990}), $W$ is a finitely  generated linear group over the real numbers. It follows that $W$ is a residually finite group, and in particular that it is Hopfian.
\end{fact}
%
%

	We fix even Coxeter systems $(W, S)$ and $(W', T)$ such that $W \equiv W'$. Let $|S| = n$, let $S = \{s_1, ..., s_n\}$, and let $N = N_W$ be the maximal size of a finite subgroup of $W$ (cf. \ref{finite_subgroups_fact}). Since $W \equiv W'$, the number $N$ is also the maximal size of a finite subgroup of $W'$. Finally, let $u_1, ..., u_m \subseteq S$ be an enumeration without repetitions of the maximal spherical subsets of $S$.	

	\begin{notation}\label{formula_notation}
	Let $\psi = \exists x_1, ..., x_n \varphi$, where $\varphi(x_1, ..., x_n)$ is the formula that says:
	\begin{enumerate}[(i)]
	\item for every $i \in [1, m]$, $\langle X_{u_i} \rangle \simeq W_{S_{u_i}}$ via the isomorphism $x_i \mapsto s_i$;
	\item for every $i \in [1, m]$, $\langle X_{u_i} \rangle$ is maximal among the subgroups of $G$ of order~$\leq N$;
	\item if $F$ is a finite subgroup of order $\leq N$, then there are $i \in [1, m]$ and $g$ such that
	 $$F \leq g\langle X_{u_i} \rangle g^{-1}.$$
\end{enumerate}
\end{notation}

	\medskip
	
		\begin{te}\label{rigidity_Coxeter} If $W \equiv W'$ are even Coxeter groups, then we have $W \simeq W'$.
\end{te}
	
	\begin{proof} Let $a_1, ..., a_n \in W'$ be witnesses of the fact that $W' \models \psi$. Let then $A  = \{a_1, ..., a_n \}$ and, for $\ell \in [1, m]$, let $G_\ell$ be the group $\langle A_{u_\ell} \rangle_{W'}$. Now, for every $\ell \in [1, m]$, $G_\ell$ is a maximal finite subgroup of $W'$ and so $G_\ell = w_\ell W'_{T_\ell} w_\ell^{-1}$ for some $w_\ell \in W'$ and maximal spherical $T_\ell \subseteq T$. We claim that $\langle T_1 \cup \cdots \cup T_m \rangle_{W'} = W'$. In fact, if this is not the case, then there is $t \in T \setminus T_1 \cup \cdots \cup T_m$ and taking a maximal spherical subset of $T$ containing $t$ we get a group which is necessarily different from all the $G_i$'s, and so either some $G_\ell$ were not maximal to start with (a contradiction), or there is a $G_{m+1}$ finite maximal in $W'$ which is not conjugate to the others, but then, $W'$ has $m+1$ non-conjugate maximal finite subgroups, which also leads to a contradiction. Hence, $T = T_1 \cup \cdots \cup T_m$ and $\langle T_1 \cup \cdots \cup T_m \rangle_{W'} = W'$. Now, for every $i \in [1, n]$ there is $\ell_i \in [1, m]$ such that $a_i \in G_{\ell_i}$. Define $t_i$ as $\pi_{T_{\ell_i}}(a_i)$, where $\pi_{T_{\ell_i}}$ is as in \ref{proj_fact} with respect to $T_{\ell_i} \subseteq T$. Notice crucially that by \ref{commuting_retraction} the choice of $\ell_i$ does not matter, and so, for $i \in [1, n]$,  $t_i$ is well-defined. Let then $T = \{t_1, ..., t_n\}$.
	
	\begin{enumerate}[$(\star_1)$]
	\item If $u \subseteq [1, n]$ is such that $S_u$ is spherical, then $\langle S_u \rangle_W \simeq \langle T_u \rangle_{W'}$ via  $s_i \mapsto t_i$. 
\end{enumerate}

\noindent We prove $(\star_1)$. Suppose that $u \subseteq [1, n]$ and $S_u$ is spherical. Let $\ell \in [1, m]$ be such that $u \subseteq u_\ell$. Then $\langle S_{u_\ell} \rangle_W \simeq \langle A_{u_\ell} \rangle_{W'}$ via the map $s_i \mapsto a_i$ and $\langle A_{u_\ell} \rangle_{W'} = G_\ell = w_\ell W'_{T_\ell} w_\ell^{-1}$ is such that for every $i \in u_\ell$ we have that $a_{i} \in G_{\ell}$. Now, for $i \in u_\ell$, let $b_i = w^{-1}_\ell a_i w_\ell \in W'_{T_\ell}$. Then clearly $\langle S_{u_\ell} \rangle_W \simeq \langle B_{u_\ell} \rangle_{W'}$ via the map $s_i \mapsto b_i$, where $B_{u_\ell} = \{b_i : i \in u_\ell\}$. Let $\pi$ be the canonical projection of $W$ onto $W'_{T_{u_\ell}}$. Then, recalling that the definition of $t_i$ did not depend on the choice of the maximal spherical subgroup to which $a_i$ belongs, for all $i \in u_\ell$ we have that:
	$$t_i = \pi(a_i) = \pi(wb_iw^{-1}) = \pi(w)\pi(b_i)\pi(w)^{-1} = \pi(w)b_i\pi(w)^{-1}.$$
Thus, $\langle B_{u_\ell} \rangle_{W'}$ and $\langle T_{u_\ell} \rangle_{W'}$ are isomorphic via the map $b_i \mapsto t_i$ as witnessed by the inner automorphism of $W'_{T_{u_\ell}}$ induced by $\pi(w) \in W'_{T_{u_\ell}}$. Thus, we proved~$(\star_1)$.

\noindent

%

	\begin{enumerate}[$(\star_2)$] 
	\item If $i, j \in [1, n]$ and $o(s_is_j)$ is infinite, then $o(t_it_j)$ is infinite.
\end{enumerate} 

\noindent We prove $(\star_2)$. Suppose not, and let $i, j \in [1, n]$ be such that $t_it_j$ is a counterexample. Let $T_* \subseteq T$ be a maximal spherical subset of $T$ such that $t_it_j \in W'_T$. Then, by $(\star_2)$, necessarily the maximal spherical $W'_T \leq W$ must be different from $W'_{T_{u_\ell}}$, for every $\ell \in [1, m]$, and this is a contradiction. This concludes the proof of $(\star_2)$.

\smallskip 
\noindent As $\langle T_1 \cup \cdots \cup T_m \rangle_{W'} = W'$, by $(\star_1)$, $\langle t_1, ..., t_n \rangle_{W'} = W'$. Now, if $(W', T)$ is a Coxeter system, then clearly $W \simeq W'$ and so we are done. If not, then the Deletion Condition (cf. \ref{deletion_condition}) fails for the pre-Coxeter system $(W', T)$ and so $W' \models \exists x_1, ..., x_n \varphi'(x_1, ..., x_n)$, where $\varphi'(x_1, ..., x_n)$ is the conjunction of $\varphi(x_1, ..., x_n)$ together with the formula saying there is a word in the variables $x_1, ..., x_n$ which fails the Deletion Condition with respect to the alphabet $\{x_1, ..., x_n\}$. But then $W \models \exists x_1, ..., x_n \varphi'(x_1, ..., x_n)$ and so replacing the role of $W$ and $W'$ and repeating the argument above (so in particular now $a_i, t_i \in W$) we have that the map $\alpha: s_i \mapsto t_i$, for every $i \in [1, n]$, is a surjective endomorphism of $W$, and so, as $W$ is Hopfian (cf. \ref{hopfian}), we have that $\alpha \in \mathrm{Aut}(W)$. Hence, modulo an automorphism, for every $i \in [1, n]$, $a_i \in s_i^W$ and so, by \ref{the_even_isomorphism_fact}, $(\langle A \rangle_W, A)$ is a Coxeter system, but this contradicts the fact that the $a_i$'s are witnesses of the fact that $W \models \exists x_1, ..., x_n \varphi'(x_1, ..., x_n)$, \mbox{recalling how $\varphi'(x_1, ..., x_n)$ was chosen.}
\end{proof}

	\begin{lemma}\label{alg_prime_lemma} Let $(W, S)$ be even, $G$ elementarily equivalent to $W$ and $\varphi(\bar{x})$ as in \ref{formula_notation}. Then for every $\bar{a} \in G^n$ such that $G \models \varphi(\bar{a})$ we have that $(\langle \bar{a} \rangle_G, \{a_1, ..., a_n \})$ is a Coxeter system (necessarily of the same type as $(W, S)$, because of how $\varphi(\bar{x})$ was defined). And so, in particular, $W$ is an algebraically prime model of its theory.
\end{lemma}
	
\begin{rk}Note that we can recover the fact that hyperbolic even Coxeter groups are first-order torsion-rigid (already proved in Section \ref{FOrigidity}, see Corollary \ref{corollary}) by combining Theorem \ref{rigidity_Coxeter}, Theorem \ref{preservation} and Remark \ref{remark_even}. Indeed, if $W'$ is a finitely torsion-generated group that is AE-equivalent to a hyperbolic even Coxeter group $W$, then $W'$ is a Coxeter group by Theorem \ref{preservation}, and it is even by Remark \ref{remark_even}, therefore $W'\simeq W$ by Theorem \ref{rigidity_Coxeter}.\end{rk}

\renewcommand{\refname}{Bibliography}

\end{document}